\documentclass[letterpaper,12pt]{amsart}
\usepackage{hyperref}
\usepackage{latexsym}
\usepackage{amssymb}
\usepackage[cp850]{inputenc}
\usepackage{epsfig}
\usepackage{amsthm}
\usepackage{amscd}
\usepackage{amsmath}
\usepackage{amsfonts}
\usepackage{graphics}
\usepackage[all]{xy}

\newtheorem{theorem}{Theorem}
\newtheorem{lemma}[theorem]{Lemma}
\newtheorem{corollary}[theorem]{Corollary}
\newtheorem{prop}[theorem]{Proposition}

\newtheorem*{theorem*}{Theorem}
\newtheorem*{corollary*}{Corollary}
\theoremstyle{definition}
\newtheorem{example}[theorem]{Example}
\newtheorem{remark}[theorem]{Remark}
\newtheorem{definition}[theorem]{Definition}
\newtheorem*{remark*}{Remark}
\newtheorem{question}[theorem]{Question}
\newtheorem{observation}[theorem]{Observation}
\newtheorem*{definition*}{Definition}
\newtheorem*{example*}{Example}

\numberwithin{theorem}{section}

\newcommand{\BC}{\mathbb C} 
\newcommand{\BR}{\mathbb R} 
\newcommand{\BN}{\mathbb N} \newcommand{\BQ}{\mathbb Q}
 \newcommand{\BZ}{\mathbb Z}
 \newcommand{\BT}{\mathbb T}

\newcommand{\CC}{\mathcal C} \newcommand{\calD}{\mathcal D}
 
\newcommand{\CG}{\mathcal G} 
 
 \newcommand{\CL}{\mathcal L}
\newcommand{\CM}{\mathcal M} 
 \newcommand{\CP}{\mathcal P}
 
\newcommand{\CS}{\mathcal S} \newcommand{\CT}{\mathcal T}

\newcommand{\aut}{\textup{Aut}(F_n)}

\newcommand{\wt}{\widetilde}
\newcommand{\wh}{\widehat}
\newcommand{\nid}{\noindent}

\newcommand{\tup}{\textup}

\DeclareMathOperator{\trace}{Tr}

\DeclareMathOperator{\rank}{rank}

\DeclareMathOperator{\Mod}{\tup{Mod}(\Sigma)}
\DeclareMathOperator{\Modp}{\tup{Mod}(\Sigma,*)}

\newcommand{\comment}[1]{}

\usepackage{etoolbox}

\makeatletter
\patchcmd{\epigraph}{\@epitext{#1}}{\itshape\@epitext{#1}}{}{}
\makeatother

\begin{document}
\title  [Homological eigenvalues]  {Homological eigenvalues of lifts of pseudo-Anosov mapping classes to finite covers}
\author   {Asaf Hadari}
\date{\today}
\begin{abstract} Let $\Sigma$ be a compact orientable surface of finite type with at least one boundary component. Let $f \in \tup{Mod}(\Sigma)$ be a pseudo Anosov mapping class. We prove a conjecture of McMullen by showing that there exists a finite cover $\wt{\Sigma} \to \Sigma$ and a lift $\wt{f}$ of $f$ such that $\wt{f}_*: H_1(\wt{\Sigma}, \BZ) \to H_1(\wt{\Sigma}, \BZ)$ has an eigenvalue off the unit circle. \end{abstract}
\maketitle

\vspace {10mm}

\section{Introduction}

Let $\Sigma$ be a compact orientable surface and let $\Mod$ be its mapping class group - the group of isotopy classes of orientation preserving diffeormorphisms from $\Sigma$ to itself the fix the boundary point wise. The finite dimensional representation theory of $\Mod$ is  a nascent field of study. These groups have extensive collections of finite dimensional representations, but many basic questions remain mysterious. For instance for most $\Sigma$ it is not known whether or not $\Mod$ is a linear group.  

The largest known collection of representations are the \emph{homological representations} which are associated to finite covers $\pi: \Sigma' \to \Sigma$. The first of these, which is associated to the trivial cover, is the standard  homological representation $\tup{Mod}(\Sigma) \to \tup{GL}(H_1(\Sigma, \BQ))$ given by the induced action on first homology. The kernel of this representation is called the \emph{Torelli group}.

More generally, fix a point $* \in \Sigma$. Let $\Modp$ be the group of orientation preserving self diffeomorphisms of the pair $(\Sigma, *)$ that fix the boundary point wise, mod isotopies that fix $*$ and the boundary point wise. The group $\Modp$ acts on $\pi_1(\Sigma, *)$ by automorphisms.  

Let $K < \pi_1(\Sigma, *)$ be a finite index subgroup, and let $\pi: \Sigma' \to \Sigma$ be the associated finite cover. Let $G_K = \{f \in \Modp | f(K) = K \}$. The group $G_K$ is a finite index subgroup of $\Modp$. We have a natural map $\rho_K: G_K \to \tup{GL}(H_1(K; \BZ))$. Topologically, every element $f \in G_K$ can be lifted to a diffeomorphism $f': \Sigma' \to \Sigma'$. The diffeomorphism $f'$ induces a map $f'_*: H_1(\Sigma'; \BQ) \to H_1(\Sigma'; \BQ)$. The transformation $f'_*$ is $\rho_K(f')$.  

The representations $\rho_K$ are called homological representations. They have been studied extensively by many authors. For example: Gr\"unewald, Larsen, Lubotzky and Malestein use these representation to construct several different infinite families of arithmetic quotients of mapping class groups (see \cite{GLLM}). In \cite{PW}, Putman and Wieland exhibit a connection between properties of homological representations and the virtual first betti number of $\Mod$. In work of Lubotzky and Meiri, and separately in work of Malestein and Souto these representation were used to describe generic properties of random elements of $\Mod$ (see \cite{LMeiri}, \cite{LMeiri2}, \cite{MalesSouto}). 

In addition to providing information about the group $\Mod$ as a whole, these representations also provide in formation about individual elements. For example, Koberda and later Koberda and Mangahas showed that the family of homological representations can detect the Nielsen-Thurston classification of a mapping class (\cite{Kober}, \cite{KoberMang}). 

It is natural to try to understand whether or not the topological invariants associated to a mapping class $f$ can be recovered from its homological representations. McMullen studied this question for pseudo Anosov mapping classes by considering the following invariant. Fix $f \in \Modp$, a pseudo-Anosov mapping class. Given a finite index subgroup $K < \pi_1(\Sigma, *)$, let $\sigma_K$ be the spectral radius of the operator $\rho_K(f)$ (that is - the modulus of its largest eigenvalue). It is a simple exercise to show that $\sigma_K$ is at least $1$ and at most $\lambda$ - the dilatation of $f$. In \cite{Mcm}, McMullen shows that if the invariant foliations of $f$ have a singularity with an odd number of prongs, then $\sup \sigma_K < \lambda $, where the supremum is taken over all finite index subgroups $K$. McMullen asked the following question, whose positive resolution has become a well known conjecture. 

\begin{question}
In the notation above, is $\sup \sigma_K >1$?
\end{question}

In a previous paper (\cite{inford}) we provided evidence for this conjecture by proving the following. 
\begin{theorem} (\textup{(Hadari)})
Suppose that $\Sigma$ has at least one boundary component. Then for any infinite order element $f \in \Mod$, there is a finite cover $\pi: \Sigma' \to \Sigma$ to which $f$ lifts to a map $f'$, such that $f'_*: H_1(\Sigma'; \BQ) \to H_1(\Sigma'; \BQ)$ has infinite order. 
\end{theorem}

\nid In this paper, we use a strategy inspired by the proof in \cite{in ford} to provide the following answer to McMullen's question. 

\begin{theorem} \label{theorem1}
Suppose that $\Sigma$ has at least one boundary component. Then for any $f \in \Mod$ with positive topological entropy, there exists a regular finite cover $\pi: \Sigma' \to \Sigma$ to which $f$ lifts to a map $f'$, such that $f'_*: H_1(\Sigma'; \BQ) \to H_1(\Sigma'; \BQ)$ has eigenvalues off of the unit circle. Furthermore, if $f$ is pseudo-Anosov, this cover can be taken to have a solvable deck group. 

\end{theorem}

\nid We also provide an analogous result for automorphisms of free groups. 

\begin{theorem} \label{theorem2} Let $n \geq 2$ and let $\overline{f} \in \tup{Out}(F_n)$ be a fully irreducible automorphism. Then there exists a representative $f$ of $\overline{f}$ and finite index subgroup $K \lhd F_n$ such that $f(K) = K$, and $f_*: H_1(K;\BQ) \to H_1(K; \BQ)$ has eigenvalues off of the unit circle. The subgroup $K$ can be taken such that $F_n/K$ is solvable.  

\end{theorem}

Note that for any surface $\Sigma$, there is a natural map $\Mod \to \tup{Out}(\pi_1(\Sigma))$. In our statement of Theorem \ref{theorem1}, we restrict ourselves to surfaces with boundary. These surfaces have free fundamental groups. Given a pseudo-Anosov mapping class on a surface with boundary, the surface $\Sigma$ is homotopy equivalent to an invariant train track graph $\Gamma \subset \Sigma$, and the map $\overline{f}$ is induced by a continuous function $\varphi: \Gamma \to \Gamma$. The map $\varphi$ fixes some point $* \in \Gamma$. Taking this point to be out base point, we  get a representative of $\overline{f}$ called a \emph{train track representative}. A similar notion exists for fully irreducible automorphisms (see \cite{FLP} for surfaces, and \cite{BeH} for $\tup{Out}(F_n)$). We will deduce theorems \ref{theorem1} and \ref{theorem2} from the following theorem, whose proof will take up the majority of this paper. 

\begin{theorem} \label{theorem3} Let $n \geq 2$ and let $\overline{f} \in \tup{Out}(F_n)$ be either a fully irreducible automorphism or the image of a pseduo-Anosov mapping class. Let $f$ be a train track representative of $\overline{f}$. Then there exists a finite index subgroup $K \lhd F_n$ such that $f(K) = K$ and $f_*: H_1(K;\BZ) \to H_1(K; \BZ)$ has eigenvalues off of the unit circle. Furthermore, we can choose $K$ such that $F_n/K$ is solvable.  
\end{theorem} 

\begin{remark}
As we were concluding the writing of this paper, Yi Liu also published an independent proof of McMullen's conjecture. Like the proof in this paper, his proof is to some extent inspired by our proof in \cite{inford} but aside from this initial inspiration the two proofs are very different. The end results are also somewhat different. Liu's proof covers the case of closed surfaces, which the proof in this paper does not (our proof fails for closed surfaces in exactly one spot - Lemma \ref{nilstab}). The proof in this paper covers the $\tup{Out}(F_n)$ case, which Liu's does not, and provides the extra information that the finite cover can be taken to be solvable. 
\end{remark}

\subsection{Strategy and organization of the proof} 
Theorem \ref{theorem3} is non-trivial only when all of the eigenvalues of $f_*$ are roots of unity. By replacing $f$ with a power of itself, we can assume that all of its eigenvalues are $1$, and in particular it has a $1$-eigenspace. 

To such an autmorphism we introduce a matrix $A_f$ which we call the \emph{equivariant Magnus matrix}. It is related to the Magnus representation of $f$ (see \cite{mag}, \cite{mag2} for definitions). The entries of this matrix are polynomials, which we view as elements of the group ring of some quotient $H_f$ of $H_1(F_n, \BZ)$.  

Given a matrix whose entries are polynomials in the variables $X^{\pm 1}_1, \ldots, X^{\pm 1}_m$, we can substitute numbers $\xi_1, \ldots, \xi_m$ for $X_1, \ldots, X_m$ to get a matrix with entries in $\BC$. This is called the \emph{specialization of the matrix at $\xi_1, \ldots, \xi_m$}. The equivariant Magnus matrix has the property that its specialization at roots of unity contain information $\rho_K(f)$ for a certain collection of abelian covers $K$. One particular connection is that if we specialize $A_f$ at roots of unity and get a matrix that has eigenvalues off of the unit circle, then $\rho_K(f)$ has eigenvalues off of the unit circle for some abelian cover $K$.

We now face the question of having to tell when a matrix with polynomial coefficients has a specialization at roots of unity with eigenvalues off of the unit circle. One possible answer is to look at the trace of such a matrix. If the trace of the matrix is in some sense large (say if the $L^2$ norm of its coefficients is greater than the dimension of the matrix) then it is possible to use the Fourier transform on abelian groups to find a specialization as required.  

In Section \ref{magnus} we introduce the equivariant Magnus matrix, discuss the connection between its stabilizations and the homological representations of $f$, and provide two criteria for finding a specialization that has eigenvalues off of the unit circle. The remainder of proof shows how to find a sequence of covers where the trace of the equivariant Magnus matrices becomes larger and larger in the sense of Section \ref{magnus}.  

In Section \ref{transition} we introduce a combinatorial object called the \emph{transition graph} which encodes a great deal of information about the $f$. In particular, attached to this transition graph is a convex polygon $\CS^e \varphi \subset H_f \otimes \BR$, which we call the \emph{equivariant shadow}, and which which we use extensively in our proof. This polygon is related to the norm ball of the Thurston norm (this is explained in Lemma \ref{dimshad}). 

Morally speaking, we expect the convex hull of the support of $\trace A_f$ to be some homothetic image of the polygon $\CS^e \varphi$ and when this is the case, $\trace A_f$ is large in the sense of Section \ref{magnus}. When this is not the case, it is due to some cancellation occurring at the vertices of this polygon. In Section \ref{stable} we discuss this cancellation, and show that for every vertex it is possible to find a nilpotent cover where it does not cancel. Finally, in Section \ref{proofs} we collect several important technical lemmas, and complete the proof of our Theorems \ref{theorem1}, \ref{theorem2}, \ref{theorem3}.

\section{The Magnus matrix and its specializations} \label{magnus}

In this section we introduce a key concept: the equivariant Magnus matrix, and two lemmas: the anchoring lemma and the $L^2$-trace lemma that will be our central tools in proving Theorem \ref{theorem3}. 

Throughout the section, let $f \in \tup{Aut}(F_n)$ be a train track representative of a pseudo-Ansov mapping class or a fully irreducible automophism. Suppose that $f$ is induced by the continuous map $\varphi: \Gamma \to \Gamma$, fixing the base point $*$.

\subsection{The $f$-equivariant torsion free universal abelian cover} 
Let $G = F_n \rtimes_f \BZ$. Consider the endomorphism $i: F_n \to G$ given by $i(w) = (w,0)$.
\begin{definition} Let $h: G \to H_1(G;\BQ)$ be the natural map, and let $U_f = h\circ i$. Let $\wt{\Gamma}_f$ be the cover corresponding to $U_f$. We call $\wt{\Gamma}_f$ the \emph{$f$-equivariant torsion free universal abelian cover of $\Gamma$}. 
\end{definition}

\nid We begin by making several simple observations. Since $[F_n, F_n] < U_f$, we have that $H_f = F_n / U_f$ is a finitely generated abelian group. Since $[G, G] \lhd G$, the definition of semidirect products gives that $f(U_f) = U_f$. Thus $f$ acts on the group $H_f$. 

There is an $f$-equivariant isomorphism between $H_f$ and the image of $F_n$ in $H_1(G; \BQ) \cong G_{ab} \otimes \BQ$, where the action of $f$ on $G$ is given by conjugation.  Since conjugation acts trivially on $G_{ab}$, we get that $f$ acts trivially on $H_f$. 
Let $f_*$ denote the automorphism induced by $f$ on H = $H_1(F_n, \BZ)$. The group $H_f$ is an abelian quotient of $F_n$ and is thus a quotient of $H$. Indeed, we have the natural identification $H_f \cong \tup{coker}(I_n - f_*)$.  Note that since $G_{ab} \otimes \BQ$ is torsion free, then so is $H_f$.

\subsection{The equivariant Magnus matrix of $f$} \label{magnusmatrix} Let $V_f = C_1(\wt{\Gamma}_f, \BC)$ be the space of simplicial $1$-chains with coefficients in $\BC$ in the cover $\wt{\Gamma}_f$. 

The group $H_f$ acts on $\wt{\Gamma}_f$ by deck transformations and thus permutes the edges of $\wt{\Gamma}_f$. This gives $V_f$ the structure of a $H_f$-module. 

Pick a spanning tree $T$ of $\Gamma$ and let $\wt{T}$ be a lift of $T$ to a tree in $\wt{\Gamma}_f$. For any vertex $v$ of $\Gamma$, let $\wt{v}$ be the lift of $v$ incident at $\wt{T}$. The action of the group $H_f$ on $\wt{\Gamma}_f$ by deck transformations gives a transitive permutation on the set of all pre images of $v$. By identifying $\wt{v}$ with the element $0 \in H_f$, we can identify every pre image of $v$ with an element of $H_f$.

 Given any oriented edge $\eta$ of $\Gamma$, the choice of the lift $\wt{T}$ gives a bijection $\iota_\eta$ between the collection of lifts of $\eta$ in $\wt{\Gamma}_f$ and $H_f$ given by reading off the label of the origin vertex of a lift. This identification induces a $H_f$-module isomorphism: 
$$(\BC[H_f])^{E(\Gamma)} \cong V_f $$

If we set $m = \#E(\Gamma$) then the above isomorphism is is given by $$(\sum a_{1,i} h_{1,i}, \ldots, \sum a_{m,i} h_{m,i} ) \to \sum_{j=1}^m \sum a_{j,i} \iota_{\eta_i} (h_{j,i})$$

\nid where $a_{j,i} \in \BC$ and $h_{j,i} \in H_f$.

Let $\varphi_f$ be the lift of $\varphi$ to $\wt{\Gamma}_f$ that fixes $\wt{*}$. Since $\wt{\varphi}$ maps edges to edge-paths in $\wt{\Gamma}_f$, it induces a map $\wt{\varphi}_*: V_f \to V_f$. We call this map the \emph{equivariant Magnus representation of $f$ on $\Gamma$}.

Because $f$ acts trivially on $H_f$,  we get that $\wt{\varphi}_*$ commutes with the action of $H_f$ on $V_f$, and thus induces a $H_f$-module homomorphism $V_f \to V_f$. 

Under the identification $V_f \cong (\BC[H_f])^{E(\Gamma)}$, this homomorphism is given by multiplication by an $m \times m$ matrix $A_f \in M_m(\BC[H_f])$, where $m = \#E(\Gamma)$. We call the matrix $A_f$ the \emph{equivariant Magnus matrix of $f$ on $\Gamma$}.

\subsection{Specializations of $A_f$ and abelian covers.}
Write $H_f \cong \BZ^d$. Viewing $\BZ^d$ as a multiplicative group, we can write $\BC[\BZ^d] \cong \BC[X_1^{\pm 1}, \ldots, X_d^{\pm 1}]$.

\begin{definition} Let $\xi: \BZ^d \to \BC^\times$ be a homomorphism. Denote $\xi_i = \xi(X_i)$. Let $t \in \BC[\BZ^d]$. Using the identification $\BC[\BZ^d] \cong \BC[X_1^{\pm 1}, \ldots, X_d^{\pm 1}]$, we can view $t$ as a rational function in the variables $X_1, \ldots, X_d$. By plugging in the number $\xi_i$ for the variable $X_i$, we get a number $t(\xi) \in \BC$, which we call the \emph{specialization of $t$ at $\xi$}.
\end{definition}

\begin{definition} If $A \in M_k(\BC[\BZ^d])$ for some $k$, and $\xi: \BZ^d \to \BC$ is a homomorphism then we can define the \emph{specialization of $A$ at $\xi$} to be the $k \times k$ matrix whose $(i,j)$-coordinate is $A_{i,j}(\xi)$.
\end{definition}

The space $V_f$ is the space of formal linear combinations of edges in $\wt{\Gamma}_f$ with finite support and coefficients in $\BC$. Let $W_f$ be the space we get by removing the finite support condition. The deck group $H_f$ acts on the edges of $\wt{\Gamma}_f$ by permutations. This action gives $W_f$ an $H_f$ module structure. 
\begin{definition}
Let $\xi: H_f \to \BC^\times$ be a homomorphism. Define: $$W_{f,\xi} = \{t \in W_f | \forall h \in H_f: h \cdot t = \xi(h) t \}$$
\end{definition}

\nid Notice that for every $\xi$, the space $W_{f,\xi}$ is a $m = \#E(\Gamma)$ dimensional vector space. Indeed, let $\{\eta_1, \ldots, \eta_m\}$ be the collection of edges of $\Gamma$, and $\{\wt{\eta}, \ldots, \wt{\eta_m}\}$  a set of $m$ preferred lifts of the edges of $\Gamma$ to $\wt{\Gamma}_f$. Any edge in $\wt{\Gamma}_f$ is the image of one of the $\wt{\eta_i}$ under a deck transformation. Given any $t \in W_{f,\xi}$ and an edge $\zeta$ of $\wt{\Gamma}_f$ such that $\zeta = h \cdot \wt{\eta_i}$, then the coefficient of $\zeta$ in $t$ is $\xi(h)$ times the coefficient of $\wt{\eta_i}$. Thus, we have an obvious identification $W_{f,\xi} = \BC^{\{\wt{\eta_1}, \ldots, \wt{\eta_m} \}}$.

Since the homomorphism $\wt{\varphi}_*$ is a $H_f$ module homomorphism, it acts on every space $W_{f,\xi}$ as a $H_f$ module homomorphism which we denote $A_{f,\xi}$. 

If $\{\wt{\eta_1}, \ldots, \wt{\eta_m}\}$ is the set described above, then every element of $W_{f,\xi}$ can be written as 

$$t = \sum_{i=1}^m \sum_{h\in H_f} a_i \xi(h) (h\cdot \wt{\eta_i}) $$

We can then define $$A_{f,\xi} (t) = \sum_{i=1}^m \sum_{h\in H_f} a_i \xi(h) (h\cdot A_f \wt{\eta_i})  $$

\begin{lemma} Under the identification $W_{f,\xi} = \BC^{\{\wt{\eta_1}, \ldots, \wt{\eta_m} \}}$, the matrix corresponding to the linear transformation $A_{f,\xi}$ is the specialization $A_f(\xi)$
\end{lemma}
\begin{proof}
Write $A_f \wt{\eta_i} = \sum_j w_{i,j} \wt{\eta_j}$, with $w_{i,j} \in \BC[H_f]$. Then $$A_f \cdot \sum_{h \in H_f} \xi(h) (h \cdot \wt{\eta_i}) = \sum_{h \in H_f} \xi(h) \sum_j (h\cdot w_{i,jj}) \wt{\eta_j}$$

\nid Switching the order of the summands and the fact that for any $t \in W_{f,\xi}$ and $h \in H_f$: $h \cdot t = \xi(h) t$ now gives the result. 
\end{proof}

Now suppose $\xi: H_f \to \BC^\times$ has finite image. Let $k =  |\xi(H_f)|$ be the size of the image group. Let $\Gamma_k \to \Gamma$ be the cover corresponding to the the kernel of the homomorphism $F_n \to H_f / k H_f$ given by reduction mod $k$. 

Let $t \in W_{f,\xi}$, and let $\eta$ be an edge in $\Gamma_k$. Given any two lifts $\wt{\eta}_1, \wt{\eta}_2$ of the edge $\eta$ to $\wt{\Gamma}_f$ the coefficients of $\wt{\eta}_1$ and $\wt{\eta}_2$ in $t$ are the same. Call this number $a_\eta$. Denote $\overline{t} = \sum_\eta a_\eta \eta \in C_1(\Gamma_k, \BC)$.

The action of $H_f / kH_f$ on $\Gamma_k$ by deck transformations induces a $H_f$ module structure on $C_1(\Gamma_k, \BC)$. The map $t \to \overline{t}$ is a $H_f$ module isomorphism. Call its image $\overline{W}_{f,\xi}$

The map $\varphi$ lifts to a map $\varphi_k$ of $\Gamma_k$. Since this map is $H_f$-equivariant, it fixes the space $\overline{W}_{f,\xi}$. Because the map $t \to \overline{t}$ is an isomorphism, the matrix giving the induced action $(\varphi_{k})_*$ on this space is $A_f(\xi)$.
\subsection{The anchoring lemma and the $L^2$-trace lemma}

\begin{definition}  Let $L$  be a lattice in $\BZ^d$. Let $t = \sum a_i h_i \in \BC[\BZ^d]$, where $a_i \in \BC$ and $h_i \in \BZ^d$. Define: $$t(L) = \sum_{h_i \in L} a_i$$
\end{definition}

\begin{definition} Led $B \in M_d(\BC[\BZ^d])$. Let $t_k = \trace[B^k]$. We say that $B$ is \emph{anchored} if there is some lattice $L$ and some integer $k$ such that $t_k(L) > d$.
\end{definition}

\begin{definition} We say that $f$ is \emph{anchored in $\Gamma$} if $A_f$ is anchored.

\end{definition}

\nid The following lemma relates specializations to lattices. 

\begin{lemma}\label{special lattice} Let $t \in \BC[\BZ^d]$, and let $L$ be a lattice. Let $N_L$ be the set of all $\xi: \BZ^d \to \BC^\times$ such that $\xi|_L = 1$. Then:

$$t(L) = \frac{1}{|N_L|} \sum_{\xi\in N_L} t(\xi) $$

\end{lemma}

\begin{proof} Since the functions $t \to t(L)$ and $t \to t(\xi)$ are linear in $t$, it's enough to prove the lemma for the case where $t$ is a monomial. Suppose $t = a h$, with $a \in \BC$ and $h \in \BZ^d$. Let $\overline{h}$ be the image of $h$ in the finite abelian group $G = \BZ^d/L$.

By definition: $ \sum_{\xi\in N_L} t(\xi) = a \sum_{\chi \in G^\times} \chi(\overline{h})$ where $G^\times$ is the group of characters of $G^\times$. 
Denote the trace of the regular representation of $G$ by $\rho_G$. Since $G$ is abelian, $\rho_G = \sum_{\chi \in G^\times} \chi$. Thus:

$$ \sum_{\xi\in N_L} t(\xi) = a \rho_G(\overline{h}) $$

The left hand side is equal to $0$ if $\overline{h} \neq e$ and $a \cdot |G| = a\cdot |G^{\times}| = a \cdot |N_L|$ if $\overline{h} = e$. Since $\overline{h} = e$ if and only if $h \in L$, this concludes the proof. 

\end{proof}

\begin{lemma}\tup{(The anchoring lemma)} \label{anchoring} If $f$ is anchored in $\Gamma$ then there exists an abelian cover $\Gamma_k \to \Gamma$ to which $\varphi$ lifts to a map $\varphi_k$ such that $(\varphi_k)_*: H_1(\Gamma_k, \BC) \to H_1(\Gamma_k, \BC)$ has eigenvalues off the unit circle. 
\end{lemma}

\begin{proof}
Let $i$ be an integer and $L$ a lattice such that $\trace(A^i_f) (L) > m$. Denote $t_i = \trace(A^i_f)$. By Lemma \ref{special lattice}, $t_i(L) = \frac{1}{|N_L|} \sum_{\xi\in N_L} t(\xi) > m$. Since the sum $\frac{1}{|N_L|} \sum_{\xi\in N_L} t(\xi)$ is an average, there exists a $\xi \in N_L$ such that $|t_i(\xi)| > m$.

By definition of $N_L$, $\xi(H_f)$ is finite (since $L < \ker \xi$). Let $\Gamma_k \to \Gamma$ be the cover constructed above. The space $\overline{W}_{f,\xi}$ is a $m$-dimensional, $\varphi_k$-invariant subspace of $C_1(\Gamma_k, \BC)$. Furthermore, we have that $|t_i(\xi)| = |\trace(A_{f,\xi}) | > m$. Thus, the map induced by $\varphi_k$ on $C_1(\Gamma_k, \BC)$ has eigenvalues off the unit circle. To conclude the proof, we need the following claim: 

Suppose that the map induced by $\varphi_k$ on $C_1(\Gamma_k,  \BC)$ has an eigenvalue with absolute value $>1$. Then the same is true for the map induced by $\varphi_k$ on $H_1(\Gamma_k, \BC)$.

To see this, let $U = C_1(\Gamma_k, \BC)$, and let $W \subset U$ be the subspace spanned by all closed paths in $\Gamma_k$. We have a natural identification $W \cong H_1(\Gamma_k, \BC)$. The space $U$ is spanned by elements of the form $e$, where $e$ ranges over all edges of $\Gamma_k$. Pick any norm $\|\cdot \|$ on $U$, and let $\lambda > 1$ be the spectral radius of the action of $\varphi_k$ on $U$. Then there exists an edge $e$ such that $$\lim \sup_{j \to \infty} \frac{1}{j} \log \|\varphi^j (e) \| = \log \lambda  $$
Denote $L_\lambda$ to be the direct sum of all generalized eigenspaces corresponding to eigenvalues with absolute value $\lambda$.  Setting $\nu_j = \frac{\varphi^j (e)}{ \|\varphi^j (e) \|}$, we have that the distance from $\nu_j$ to $L_\lambda$ goes to $0$ as $j \to \infty$.
Note that $\varphi^j(e)$ corresponds to a path in $\Gamma$, and any such path can be closed to a loop using a bounded number of edges. Thus, the distance of $\nu_j$ from $W$ goes to $0$ as $j \to \infty$.  Therefore $L_\lambda \cap W \neq \{0 \}$. Since this space is $\varphi_k$-invariant, it contains an eigenvector with eigenvalue of absolute value $\lambda$.

\end{proof}

\nid Given $t = \sum a_i h_i \in \BC[\BZ^d]$, denote by $\| t\|_2 = \sqrt{\sum_i a_i^2}$.

\begin{lemma} \tup{(The $L^2$-trace lemma)} \label{l2trace}
If there exists an $i$ such that $\|\trace(A_f^i) \|_2 > m$ then there exists an abelian cover $\Gamma_k \to \Gamma$ to which $\varphi$ lifts to a map $\varphi_k$ such that $(\varphi_k)_*: H_1(\Gamma_k, \BC) \to H_1(\Gamma_k, \BC)$ has eigenvalues off the unit circle. 

\end{lemma}

\begin{proof}
Let $i$ be the number given in the statement of the theorem, and let $t = \trace(A_f^i)$. For any homomorphism $\xi: H_f \to \BC^\times$, we defined the specialization of $t$ at $\xi$, $t(\xi)$. Let $H_f^*$ be the set of all $\xi$ such that $|\xi(h)| = 1$ for all $h \in H_f$. The function $\wh{t}: H_f^* \to \BC$ given by $\xi \to t(\xi)$ is called the \emph{Fourier transform of $t$}.

Setting $H_f \cong \BZ^d$, and considering $\BZ^d$ as a multiplicative group in $X_1^{\pm 1}, \ldots, X_d^{\pm 1}$, we think of $t$ as a rational function in $X_1, \ldots, X_d$ and $\xi$ as a a $d$-tuple $(\xi_1, \ldots, \xi_d)$ all of whose coordinates have modulus $1$. The Fourier transform $\wh{t}$ is the function that plugs in the $d$-tuple $(\xi_1, \ldots, \xi_d)$ into the rational function $t$. Thus, $\wh{t}$ is a continuous function from the torus $\BT^d \cong H_f^*$ to $\BC$. 

By Plancharel's theorem, $\| t\|_2 = \|\wh{t} \|_2$ where the right hand norm is the norm in $L^{2}(\BT^d)$ measured using the Haar measure of the torus. By our assumption, we have that $\| \wh{t}\|_2 > m$. Hence, there is a point $\xi \in H_f^*$ such that $|t(\xi)| > m$. Since $\wh{t}$ is a continuous function, this point can taken to coordinates that are all roots of unity. This means that $\xi(H_f)$ is finite. We now proceed exactly as in the anchoring lemma (\ref{anchoring}) to complete the proof. 

\end{proof}

\section{The transition graph of $\varphi$} \label{transition}

\subsection{The transition graph and associated objects}

The transition graph is a technical gadget that we use to encode information about the map $\varphi$. 
\begin{definition}
Let $E(\Gamma) = \{e_1, \ldots, e_m\}$ be the edge set of $\Gamma$. Pick, once and for all, an orientation on each edge of $\Gamma$. Construct a directed graph $\CT = \CT[\Gamma, \varphi]$ called the \emph{transition graph of $\varphi$} in the following way. The vertex set of $\CT$ is $\{e_1, \ldots e_m\}$. Connect the vertex $e_i$ to the vertex $e_j$ with $e(i,j)$ directed edges, where $e(i,j)$ is the number of times $\varphi(e_i)$ traverses the edge $e_j$ (in either direction). 

\end{definition}

\nid We will associate several useful objects to the graph $\CT$ which we discuss in this section.
\begin{definition} Pick a \emph{decorating function} $\textbf{d} : E(\CT) \to \BN$ such that whenever $E_{i,j}$ is the set of edges in $\CT$ from $e_i$ to $e_j$, then $\textbf{d}|_{E_{i,j}}$ is a bijection onto the set $\{1, \ldots, e(i,j) \}$. We think of each of the edges $\eta$ connecting $e_i$ to $e_j$ as corresponding to the $\textbf{d}(\eta)$-th time that $\varphi(e_i)$ traverses $e_j$.  
\end{definition} 

\begin{definition}
Extend the decorating function to a function $\textbf{d}: \CP(\CT) \to \BN$, where $\CP(\CT)$ is the set of edge paths in $\CT$, by  requiring that $\textbf{d}$ restricted to the set of all paths connecting $e_i$ to $e_j$ of length $k$ is a bijection onto $\{1, \ldots, e(i,j;k)\}$ where $e(i,j;k)$ is the number of times $\varphi^k(e_i)$ traverses $e_j$, in either direction. We think of each of a path $\eta_1 \ldots \eta_k$ connecting $e_i$ to $e_j$ as corresponding to the $\textbf{d}(\eta_1 \ldots \eta_k)$-th time that $\varphi^k(e_i)$ traverses $e_j$.
\end{definition}

\begin{definition} Let $\CP(\Gamma)$ be the set of paths in $\Gamma$. Define a \emph{path function} $\textbf{p}: \CP(\CT) \to \CP(\Gamma)$ in the following way. Let $\eta_1 \ldots \eta_k$ be a path in $\CT$ connecting $e_1$ to $e_j$. We denote $\varphi^k(e_i) = abc$ where $a,b,c \in \CP(\Gamma)$, and $b$ the path of length $1$ traversing $e_j$ that corresponds to $\textbf{d}(\eta_1 \ldots \eta_k)$. We define $\textbf{p}(\eta_1 \ldots \eta_k) = a$ if $b$ traverses $e_j$ in the positive direction and $\textbf{p}(\eta_1 \ldots \eta_k) = ab$ if $b$ traverses $e_j$ in the negative direction. Note that this convention assures that the endpoint of $\textbf{p}(\eta_1 \ldots \eta_k)$ is the initial point of $e_j$.
\end{definition}

\begin{definition}
Let $\pi: \Gamma_0 \to \Gamma$ be a regular cover to which we can lift $\varphi$ to a map $\varphi_0$. Denote the deck group of this cover by $\calD$. Let $V = V(\Gamma)$. Choose a lift $V_0$ of the set $V$ to $\Gamma_0$. Every vertex $w \in V(\Gamma_0)$ satisfies $w = \sigma(v)$ for some $v \in V_0, \sigma \in \calD$. We say that $\sigma$ is the \emph{address} of $w$, and write $\sigma = \textbf{a}(w)$.
\end{definition}

\begin{definition} Given a regular cover $\pi: \Gamma_0 \to \Gamma$ as above and lift $V_0$ of $V$, define a function $\textbf{t}_\pi: \CP(\CT) \to \calD$ called a \emph{translation function} by setting: $\textbf{t}_\pi (\eta_1 \ldots \eta_k) = \textbf{a}(w_2) \textbf{a}(w_1)^{-1}$ where $w_1, w_2$ are respectively the initial and terminal points of $\textbf{p}(\eta_1 \ldots \eta_k)$. We will most often be concerned with the translation function for the cover $\pi: \wt{\Gamma}_f \to \Gamma$. In this case, we will omit the subscript $\pi$. 
\end{definition}

\begin{definition} Define a \emph{sign function} $\textbf{s}: E(\Gamma) \to \{\pm1 \}$ in the following way. Let $\eta$ be an edge connecting $e_i$ to $e_j$. Set $\textbf{s}(\eta) = 1$ if and only if the $\textbf{d}(\eta)$-th time $\varphi(e_i)$ traverses $e_j$ is in the positive direction. We extend the definition of $\textbf{s}$ to edge paths in $\CT$ by setting: $\textbf{s}(\eta_1 \ldots \eta_k) = \textbf{s}(\eta_1) \ldots \textbf{s}(\eta_k)$.
\end{definition}

\begin{example}
Suppose $\Gamma = S^1 \vee S^1$ is the rose on two petals. The fundamental group of this graph is $F_2 = \langle a,b\rangle$ where $a$ and $b$ are loops about the two petals based at their intersection. Consider the inner autmorphism $a \to bab^{-1}$, $b \to b$. This automorphism is induced by a function $\varphi: \Gamma \to \Gamma$ which is the identity on the loop $b$ and sends $a$ to $bab^{-1}$. The transition graph $\CT$ has two vertices: one labeled $a$ and one labeled $b$. Identify $H_1(\Gamma; \BZ) \cong \BZ^2$. The vertex $b$ has one outgoing edge $\eta_1$ which connects it to itself. We have that $\textbf{s}(\eta_b) = 1$, $\textbf{p}(\eta_1)$ is the trivial path, and $\textbf{t}(\eta_1) = 0$. The vertex $a$ has three outgoing edges. One, $\eta_2$ connects it to itself and two $\eta_3, \eta_4$ connect it to $b$. We have that $\textbf{s}(\eta_2) = 1$, $\textbf{p}(\eta_2)$ is the path $b$ and $\textbf{t}(\eta_2) = (0,1)$. Of the edges connecting $a$ to $b$ we have that $\textbf{s}(\eta_3) = 1$, $\textbf{s}(\eta_4) = -1$, $\textbf{p}(\eta_1)$ is the trivial path, $\textbf{p}(\eta_4) = bab^{-1}$, $\textbf{t}(\eta_3) = 0$ and $\textbf{t}(\eta_4) = (1,0)$.
\end{example}

\nid The above definitions allow us to give a different description of the matrix $A_f$ defined in section \ref{magnusmatrix}.

\begin{observation} Let$ 1 \leq i,j, \leq m$. Let $E_{i,j}$ be the set of edges in $\CT$ connecting $e_i$ to $e_j$. Let $\textbf{t}$ be the translation corresponding to the cover $\wt{\Gamma}_f \to \Gamma$. Then:

$$(A_f)_{i,j} = \sum_{\eta \in E_{i,j}} \textbf{s}(\eta) \textbf{t}(\eta) $$ 

\nid where $\textbf{t}(\eta)$ is understood as an element of the group ring of $H_f$ supported at one point. 

\end{observation}

\nid This observation can be seen by using the definition of $A_f$ to calculate $A_f \wt{e}$, where $\wt{e}$ is a lift of the edge $e$ in $\Gamma$.

\subsection{Vertex subgraphs and extremal subgraphs of $\CT$}

The graph $\CT$ has important subgraphs, which we call \emph{extremal subgraphs} and \emph{vertex subgraphs} that play a major role in our proof. Before we define them, we require an observation, which follows from the fact that $f$ acts trivially on $H_f$. 

\begin{observation} \label{pathhomomorphism}The map $\textbf{t}: \CP(\CT) \to H_f$, where $\CP(\CT)$ is viewed as a groupoid under concatenation, is a homomorphism of groupoids. 
\end{observation}

\nid Let $\textbf{t}$ be the translation function corresponding to the cover $\wt{\Gamma}_f \to \Gamma$.
\begin{definition}

For any path $\overline{\eta} = \eta_1 \ldots \eta_k$, define $\textbf{t}_n (\overline{\eta})$, the \emph{normalized translation of $\overline{\eta}$} to be $\textbf{t}_n(\overline{\eta}) = \frac{1}{k}\textbf{t}(\eta) \in H_f \otimes \BR$.
\end{definition}

\begin{definition}

A \emph{based cycle} in $\CT$ is closed path. A \emph{cycle} is the equivalence class of a based cycle, under the relation identifying two based cycles that differ by a cyclic permutation of their edges. One corollary of Observation \ref{pathhomomorphism} is that the function $\textbf{t}$ is well defined on cycles.   

Let $\CC$ be the set of cycles in $\CT$ and let $\CC_s$ be the set of simple cycles in $\CT$ (a cycle is simple if it gives an embedding on $S^1$ into $\CT$). As a corollary of observation \ref{pathhomomorphism}, we get that $\textbf{t}_n(\CC)$ is contained in the convex hull of $\textbf{t}_n(\CC_s)$. Since $\CC_s$ is a finite set, this convex hull is a polytope. We call this polytope the \emph{equivariant shadow of $\varphi$} and denote it $\CS^e \varphi$.

Every vertex $u$ is in $H_f \otimes \BQ$. Since $\CS^e \varphi^k = k \CS^e \varphi$,  by replacing $f$ with some power of itself we can assume that every vertex of $\CS^e \varphi$ has integer vertices. In our proof of Theorem \ref{theorem3} we will show that it suffices to prove the theorem for $f^k$ for some integer $k$. Therefore, we can and will assume in the sequel that every vertex of $\CS^e \varphi$ is integral. 

\end{definition}

\begin{definition} Let $\omega: H_f \otimes \BR \to \BR$ be a linear transformation. Let $M_\omega$ be the maximal value $\omega$ takes on $\CS^e \varphi$. Let $\CT_\omega$ be the union of all $\gamma \in \CC$ such that $\omega(\textbf{t}_n(\gamma)) = M_\omega$. We call the graph $\CT_\omega$ \emph{the extremal subgraph of $\CT$ corresponding to $\omega$}. 
\end{definition}

Since every vertex of a convex polytope is the maximal set of some linear function, we have a special kind of extremal subgraph called a \emph{vertex subgraph}. 

\begin{definition}
 Let $u \in H_f \otimes \BR$ be a vertex $\CS^e \varphi$. Let $\CT_{u}$ be the union of all $\gamma \in \CC$ such that $\textbf{t}_n(\gamma) = u$. We call $\CT_{u}$ the \emph{vertex subgraph corresponding to $u$}.

\end{definition}

\begin{lemma} \label{extremegraph} Let $\omega: H_f \otimes \BR \to \BR$ be as above. Let $\gamma \in \CC$. Then $\omega(\textbf{t}_n(\gamma)) = M_\omega$ if and only if $\gamma$ is a cycle in $\CT_\omega$.
\end{lemma}

\begin{proof} The only if direction is just the definition of the subgraph $\CT_\omega$. We will prove the if direction. For any path $\delta$ in $\CT_\omega$, define $g(\delta) = \omega \textbf{t}_n(\delta)$. Suppose that $\gamma$ is a cycle in $\CT_\omega$ with $g(\gamma) < M_\omega$. 

Choose a based cycle $\eta_1 \ldots \eta_k$ in the equivalence class of $\gamma$. For every $1 \leq i \leq k$, the edge $\eta_i$ is part of the graph $\CT_\omega$, and hence is contained in some cycle whose normalized translation is $M_\omega$. If $\eta_i$ connects vertices $e$ to $e'$, we can thus find a path $\zeta_i$ connecting $e'$ to $e$ such that $g(\eta_i \zeta_i) = M_\omega$. Let $l_i$ be the length of $\zeta_i$. 

Let $\zeta$ be the path $\zeta_k \ldots \zeta_1$. Since $g(\eta_i \zeta_i) = M_\omega$, we have that: $$\omega(\textbf{t}(\zeta_i)) = (l_k+1)M_\omega - \omega(\textbf{t}(\eta_i)).$$

Let $l =l_1 + \ldots + l_k$. Our assumption that $g(\gamma) < M_\omega$ gives that:

$$g(\zeta) = \frac{1}{l}\sum_{i=1}^k (l_k+1)M_\omega - \omega(\textbf{t}(\eta_i)) = (1 + \frac{k}{l})M_\omega - \frac{k}{l}g(\gamma) > M_\omega$$

This is a contradiction to our definition of $M_\omega$.

\end{proof}

\begin{observation} \label{extremegraphgeneral} 
The notion of vertex subgraphs is central to our proof, and we will need to use it in a more general context than the one outlined above. Note that the proof of Lemma \ref{extremegraph} did not use any properties of $\textbf{t}$, aside from the fact that it is homomorphism from the groupoid of paths to an abelian group. Thus we can define extremal and vertex graphs with respect to any such homomorphism. We can take this a step further. Any function from cycles to an abelian group that is additive on based cycles which are based at the same point can be extended to a homomorphism from the groupoid of paths. Thus, even in this more general situation we can still define extremal and vertex subgraphs. 
\end{observation}

\begin{definition} Let $u$ be a vertex of $\CS^e \varphi$. Let $e_i, e_j$ be vertices of $\CT$. Let $E_{i,k}(u)$ be the set of edges in $\CT_u$ connecting $e_i$ to $e_j$. Define the \emph{vertex matrix of $u$} or $A_{f,u}$ by setting: 
$$(A_{f,u})_{i,j} = \sum_{\eta \in E_{i,j}(u)} \textbf{s}(\eta) \textbf{t}(\eta) $$

\nid Similarly, we can define a matrix for any subgraph of $\CT$. 
\end{definition}

\begin{definition} The vertex $u$ is said to be \emph{stable} if $A_{f,u}$ is not nilpotent. 
\end{definition}

\subsection{Subgraphs and covers} Let $\pi: \Gamma' \to \Gamma$ be a cover to which $\varphi$ can be lifted to a map $\varphi'$. The transition graph of $\varphi'$ is a cover of the graph $\CT$. We will denote it by $\CT[\pi]$. If $\CT_0$ is a subgraph of $\CT$, we will denote it by $\CT_0[\pi]$. Given an vertex subgraph $\CT_u \subset \CT$, we denote its lift to $\CT[\pi]$ by $\CT_u[\pi]$, and let $A_{f,u}[\pi]$ be the associated matrix. We say that $u$ is \emph{stable in the cover $\pi$} if $A_{f,u}[\pi]$ is not nilpotent.

\subsection{The dimension of $\CS^e \varphi$}  
The group $G = F_n \rtimes_f \BZ$ is the fundamental of a mapping torus $M_f$. If $f: \Sigma \to \Sigma$ is a surface diffeomorphism, then this mapping torus is a $3$-manifold. If $f$ is a free group automorphism then we form the mapping torus $M_f = \Gamma \times I / \sim$ where $(x,0)\sim (\varphi(x), 1)$. 

We can write $H_1(M_f; \BZ) \cong H^e_1(\Gamma; \BZ) \oplus \BZ$ where  $H^e_1(\Gamma; \BZ)$ is the image of $H_1(\Gamma; \BZ)$ in $H_1(M_f; \BZ)$.  Let $\gamma$ be a cycle in $\CT$. Following Fried (who used an equivalent definition), we call $(\textbf{t}_n(\gamma), 1) \in H_1(M_f;\BR)$ a \emph{homological direction}. For $f$ pseudo-Anosov and $\Sigma$ compact, Fried studied the cone on all homological directions and related it to the Thurston norm. 

Given a $3$-manifold $M_f$ that fibers over the circle,  Thurston defines in \cite{thurstnorm} a semi norm $\tau$ on $H_2(M_f, \partial M_f ;\BR)$. 

The corresponding norm on $H_2(M_f, \partial M_f ;\BR) / \ker \tau$ is a convex polytope. One of the top dimensional faces of this polytope is called the \emph{fibered face}. If $f$ is pseudo-Anosov, and $\Sigma$ has $b \geq 1$ boundary components then $$\dim H_2(M_f, \partial M_f ;\BR) / \ker \tau = \dim H_2(M_f, \partial M_f ;\BR) - (b-1)$$ 

Let $\CC$ be the cone on the set of homological directions. In \cite{Friedzeta}, Fried proves that this cone has the same dimension as a cone on the fibered face (in fact, he proves a stronger claim - the two cones are dual). By Lefschetz duality and the universal coefficients theorem, $\dim H_2(M_f, \partial M_f; \BR) = \dim H_1(M_f; \BR)$. So, for a surface diffeomorphism we get that $\dim \CS^e \varphi = \dim H_f \otimes \BR + 1 - b$.  A simpler statement holds for the case where $f \in \tup{Out}(F_n)$. In this case, in \cite{DLK2} Dowdall, Leininger and Kapovich prove that the cone $\CC$ is $\dim H_1(M_f; \BR)$ dimensional. This is also proved separately by Algom-Kfir, Hironaka, and Rafi in \cite{YEK}. This means that $\dim \CS^e \varphi$ is $\dim H_f \otimes \BR$. We summarize this discussion in the following lemma. 

\begin{lemma} \tup{Dimension of $\CS^e \varphi$} \label{dimshad} 
If $f: \Sigma \to \Sigma$ is a pseudo Anosov mapping class and $\Sigma$ has $b \geq 1$ boundary components then $\dim \CS^e \varphi = \dim H_f \otimes \BR + (1-b)$. If $f\in \tup{Out}(F_n)$ is fully irreducible then $\dim \CS^e \varphi = \dim H_f \otimes \BR $. 

\end{lemma}

\section{Stabilizing vertex subgraphs} \label{stable}

\nid Our goal in this section is to describe a process we call vertex stabilization, in which we start with a vertex of $\CS^e \varphi$ and find a cover in which it is stable. Our method uses properties of nilpotent groups.

\subsection{Nilpotent groups} Let $G$ be a finitely generated group. Define $G_1 = G$, and for every $i$ set $G_{i+1} = [G,G_i]$. The group $G$ is said to be \emph{nilpotent} if $G_n$ is trivial for some value of $n$. The sequence of subgroups $G_i$ is called the \emph{lower central series of $G$}. 

In \cite{resprop}, Koberda introduces a modified form of the lower central series called the \emph{torsion free lower central series}. This is a series of the form $$\ldots \lhd G_3^{TF} \lhd G_2^{TF} \lhd G_1^{TF} = G$$ such that:

\begin{enumerate}
\item The groups $G^{TF}_i$ are characteristic in $G$. 
\item The groups $N_i = G/G^{TF}_i$ are nilpotent. 
\item The groups $L_i = G^{TF}_i/G^{TF}_{i+1}$ are finitely generated torsion free abelian groups that are central in $N_{i+1}$. 
\end{enumerate} 

Koberda shows that if $G$ is of the form $G = F \rtimes \BZ$ where $F$ is a surface group or a free group then $\bigcap G^{TF}_i = \{e\}$.

\subsection{Nilpotent stabilization}
Let  $G = F_n \rtimes_f \BZ$, and let $i: F_n \to G$ be given by $i(w) = (w,0)$. Let $G^{TF}_1, \geq G^{TF}_2 \geq \ldots$ be the torsion free lower central series of $G$. For every $j \geq 1$, set $K_j = i^{-1}(G^{TF}_j) \lhd F_n$, $N_j = F_n / K_j$ and $L_j = K_j/ K_{j+1}$.\\

\nid Let $\pi_j$ be the cover of $\Gamma$ corresponding to $N_j$. Denote the corresponding translation function by $\textbf{t}_j$. The subgroups $K_j$ are all $f$-invariant and thus $f$ acts on the groups $N_j$ and $L_j$. Since $f$ acts trivially on $H_f$, it is a standard fact that it acts trivially on each $L_j$.

\begin{definition} A subgraph $\CT' \subseteq \CT$ is called \emph{j-stable} if it has non-trivial cycles and for infinitely many $k$ there exists a $p$ such that:
$$\sum_\gamma \textbf{s}(\gamma) \textbf{t}_j (\gamma) \neq 0 \in \BC[N_j] $$

\nid where the sum is taken over all based cycles of length $k$ in $\CT'$ based at $p$. A vertex $u$ of $\CS^e \varphi$ is said to be $j$-stable if its vertex subgraph is $j$-stable. 
\end{definition}

\nid Since $N_1 = H_f$, saying that the vertex $u$ is stable is equivalent to saying that it is $1$-stable.

\begin{definition} A subgraph $\CT' \subseteq \CT$ is called \emph{$j$-consistent} if for every vertex $p$ of $\CT'$ there exists $x_p \in N_j$ and an integer $d$ dividing the lengths of all cycles in $\CT'$ such that  for any cycle $\gamma$ of length $k$ based at $p$: 

$$\textbf{t}_j(\gamma) = f^{k}(x_p)  \cdot \ldots \cdot f^{3d}(x_p) f^{2d}(x_p) f^{d}(x_p).$$
\end{definition}

\nid Note that a vertex subgraph is an example of a $1$-consistent subgraph.

\begin{lemma} \label{consistentsubgraph} Let $\CT' \subseteq \CT$ be a $j$-consistent subgraph. Then there exists a $(j+1)$-consistent subgraph $\CT'' \subseteq \CT'$. Furthermore, the subgraph $\CT''$ has the property that if $\CT''$ is $(j+1)$-stable then $\CT'$ is also $(j+1)$-stable. \end{lemma}

\begin{proof}

Let $x_p \in N_j$ be the elements provided by the definition of $j$-consistency. For each $p$, pick $y_p \in N_{j+1}$ whose image in $N_j$ is $x_p$. For any $l$ divisible by $d$, let $P_l(p) = f^l(y_p) \cdot \ldots \cdot f^d(y_p)$.

Let $\gamma_p$ be a based cycle  of length $k$ in $\CT'$, that is based at the vertex $p$. The \emph{deviation of $\gamma_p$} or $\Delta(\gamma_p)$ is given by the equation: 

$$\Delta(\gamma_p) = P_k^{-1}(p)\textbf{t}_{j_+1}(\gamma_p) $$

For any path $\delta = \delta_1 \delta_2$ in the graph $\Gamma$ we have that $\varphi(\delta) = \varphi(\delta_1)\varphi(\delta_2)$.  Thus, for any path $\delta$ in the graph $\CT$, and any edge $\eta$ whose initial point is the endpoint of $\delta$ we have that $\textbf{p}(\delta \eta) = \varphi(\delta) \textbf{p}(\eta)$. More generally, if $\delta'$ is a path of length $l$ whose initial point is the endpoint of $\delta$ then $\textbf{p}(\delta \delta') = \varphi^l(\delta) \textbf{p}(\delta')$.

Now let $\beta_p$ be a based cycle of length $l$ in $\CT'$ that is also based at $p$.  By the previous paragraph:

$$\textbf{t}_{j+1}(\gamma_p \beta_p) = f^l(\textbf{t}_{j+1}(\gamma_p)) \textbf{t}_{j+1}(\beta_p) = f^{l}(P_k(p) \Delta(\gamma_p))P_l(p) \Delta(\beta_p) $$

\nid Since $f$ acts trivially on $L_j$, and $L_j$ is central in $N_{j+1}$, we get: $\textbf{t}_{j+1}(\gamma_p)\beta_p = P_{k+l}(p)\Delta(\gamma_p) \Delta(\beta_p)$.  Thus, we have that: $$\Delta(\gamma_p \beta_p) = \Delta(\gamma_p) \Delta(\beta_p).$$ 

Since $L_j$ is an abelian group, it follows from the above calculation that if $\beta_p$ is obtained from $\gamma_p$ by cyclic reordering such that both are cycles based at $p$, then $\Delta(\gamma_p) = \Delta(\beta_p)$.

Let $M$ be the vector space  $L_j^{V(\CT)} \otimes \BR$. Given an un-based cycle $\gamma$ in $\CT'$, define the \emph{base point free deviation of $\gamma$} or $\overline{\Delta}(\gamma)$ to be the following element of $M$. Set $\overline{\Delta}(\gamma)[p] = 0$, if $\gamma$ doesn't pass through $p$. Otherwise, set $\overline{\Delta}(\gamma)[p]$ to be  the image of $\Delta(\gamma_p)$ in $L_j \otimes \BR$, where $\gamma_p$ is a basing of the loop $\gamma$ at $p$. Note that this function depends on $j$. We say that $\overline{\Delta}$ is the $j$-level base point free deviation function. 

The function $\overline{\Delta}$ is additive on cycles in $\CT'$. Furthemore, since $G$ is finitely generated, the vector space $L_j \otimes \BR$ is finite dimensional, and hence $M$ is finite dimensional. By Observation \ref{extremegraphgeneral}, we can use the map $\overline{\Delta}$ to choose a vertex subgraph $\CT'' \subseteq \CT'$ corresponding to the vertex $v \in M$. 

Since $L_j$ is torsion free, we have an inclusion $V(\CT')^{L_j} \subset M$. Fix a vertex $p$ of $\CT$ and let $v_p$ be the $p$-coordinate of $v$. For any cycle $\gamma_p$ based at $p$ of length $k$ we have that $\textbf{t}_j(\gamma_p) = P_k(p) \Delta_p(\gamma_p)$. We must therefore have that $kv \in L_j$. If we write $v_p = \frac{1}{q} w_p$ with $w_p \in L_j$ and $q \in \BN$, we get that $q | k$.

Set $d'$ to be the greatest common divisor of the lengths of all loops in $\CT''$. We can take $w_p = d' v$. We have that $d' = ad$ for some integer $a$.  
Define: $z_p = f^{ad}(y_p) \cdot \ldots \cdot f^d(y_p) \cdot w_p$. This choice of $z_p$'s makes the graph $\CT''$ consistent. Indeed, for any cycle $\gamma_p$ of length $k = bd'$ we have that: 

$$\textbf{t}(\gamma_p) = P_k(p) \Delta_p(\gamma_p) = P_k(p) k v_p = P_k(p) b w_p = f^{k}(z_p) \cdot \ldots \cdot f^{d'}(z_p)   $$

\nid Now suppose that $\CT''$ is $j+1$-stable. Since it it is vertex subgraph of $\CT'$, by Observation \ref{extremegraphgeneral} and Lemma \ref{extremegraph} we have that $\CT'$ is $(j+1)$-stable.

\end{proof}

\begin{definition}
A subgraph $\CT' \subseteq \CT$ is called a \emph{$j$-vertex subgraph} if there is a sequence of subgraphs $\CT' = \CT_j \subseteq \ldots \subseteq \CT_1$ such that for all $i$, $\CT_{i+1}$ is $(i+1)$-consistent and a vertex subgraph in $\CT_{i}$ with respect to the $i$-level map $\overline{\Delta}$. Note that if $\CT_{i+1}$ is a $i$-vertex subgraph then so is $\CT_l$ for any $l < i$.

\end{definition}

\begin{lemma}  \tup{(Nilpotent stability)} \label{nilstab} Let $u$ be a vertex of $\CS^e \varphi$. There exists a $j \geq 1$ such that $u$ is $j$-stable. 
\end{lemma}

\begin{proof}
By repeated application of Lemma \ref{consistentsubgraph}, we can find a for every $j$ a $j$-vertex subgraph $\CT_{u,j} \subseteq \CT_u$.

Since $\CT_u$ is a finite graph, any sequence of subgraphs must stabilize, say at $\CT_{u,N}$. That is - for every $k > N$, $\CT_{u,k} = \CT_{u,N}$. Given two based loops $\beta, \gamma$ in $\CT_N$ based at the same point and of the same length we must have that $\textbf{t}_k(\gamma) = \textbf{t}_k(\beta)$ for every $k \geq N$ (otherwise we could choose a further vertex subgraph).  

Let $\textbf{t}_\infty$ be the translation function corresponding to the universal cover of $\Gamma$. Since $\varphi$ is a train track representative, for any edge $e$ of $\Gamma$ the path $\varphi^l(e)$ is immersed in $\Gamma$. Since $\pi(\Gamma)$ is a free group, this means that given two different based cycles $\gamma_1$, $\gamma_2$ in $\CT$ of the same length and based at the same point we must have that $\textbf{t}_\infty(\gamma_1) \neq \textbf{t}_\infty(\gamma_2)$. Since the sequence $\{K_j\}_j$ satisfies that $\bigcap K_j = \{1\}$ the graph $\CT_{u,N}$ is a disjoint collection of cycles. Any such collection is obviously $N$-stable. The result now follows.

\end{proof}

\subsection{Upgrading nilpotent stabilization}

\subsubsection{The trace of powers lemma}
Fix $r,l \in \BN$ and $A \in M_l(\BC[\BZ^r])$. For any integer $k$, let $t_k = \trace[A^k]$. 

\begin{lemma} \tup{(The trace of powers lemma)} \label{top} Given any lattice $L < \BZ^r$ there exists a collection $\alpha_1, \ldots, \alpha_s \in \BC$ (which depends on $L$) and a number $C>0$ such that $$t_k(L) = C \sum_i \alpha_i^k. $$

\nid Furthermore, if $A$ is not nilpotent then there exists a number $N > 0$, such that for all $j > N$,  $t_k(j\BZ^r) \neq 0$ for infinitely many values of $k$.
\end{lemma}

\begin{proof} 
Fix a lattice $L$. Recall from the proof of Lemma \ref{anchoring} that there exists a finite set $N_L \subset (\BC^\times)^r$ such that for any $k$:
$$t_k(L)  =\frac{1}{|N_L|}\sum_{\xi \in N_L} t_k(\xi). $$
Set $p(x) = \det(xI_r - A) \in \BC[\BZ^r][x]$. We think of $p$ as a polynomial in the variable $x$ with coefficients in $\BC[\BZ^r]$. Let $\xi: \BZ^r \to \BC^\times$. The map $\xi$ extends linearly to a ring homomorphism $\BC[\BZ^r][x] \to \BC[x]$. The image of $p$ under this homomorphism is called the \emph{specialization of $p$ at $\xi$} and is denoted $p_\xi$. Note that $p_\xi$ is a degree $r$ polynomial.

Pick $r$ (not necessarily continuous) functions $\rho_1, \ldots, \rho_r: (\BC^\times)^r: \to \BC$ such that for any $\xi$: $\rho_1(\xi), \ldots, \rho_r(\xi)$ is the collection of all roots of the polynomial $p_\xi$, counted with multiplicity. 

For any $\xi$, the numbers $\rho_1(\xi), \ldots, \rho_r(\xi)$ are the roots of the characteristic polymomial of the matrix $A(\xi)$. Thus, $\rho^k_1(\xi), \ldots, \rho^k_r(\xi)$ are the roots of the characteristic polynomial of the matrix $A^k(\xi)$. Therefore, $t_k(\xi) = \sum_i \rho^k_i(\xi)$. It follows that: 

$$t_k(L) = \frac{1}{|N_L|}\sum_{\xi \in N_L} \sum_i \rho_i^k(\xi). $$

This shows the first claim of the lemma.  We now show the second  claim. Since $A$ is not nilpotent, $t_k \neq 0$ for some value of $k$. Recall that the Fourier transform $\wh{t_k}$ is the restriction of the function $\xi \to t_k(\xi)$ to $(S^1)^r$, where $S^1 = \{z \in \BC: |z| = 1 \}$. This is a continuous function since $t_k$ has finite support. Since $t_k \neq 0$, this function is not the zero function. 

Denote $N_j = N_{j \BZ^r}$. The sets $N_j$ become equidistributed in the torus $(S^1)^r$ as $r \to \infty$. In particular, for all sufficiently large $j$, there exists $\xi \in N_j$ and $1 \leq i \leq r$ such that $\rho_i(\xi) \neq 0$. The second claim now follows from the elementary fact that if $\alpha_1, \ldots, \alpha_s$ are not all $0$, then $\sum \alpha_i^k \neq 0$ for infinitely many values of $k$.
\end{proof}

Note that an identical proof holds if we replace the lattice $L$ with a translate of itself. We get the following. 

\begin{lemma} Given any lattice $L < \BZ^r$ and a vector $\overline{w} \in \BZ^r$ there exists a collection $\alpha_1, \ldots, \alpha_s \in \BC$ (which depends on $L$) and a number $C>0$ such that $$t_k(L + \overline{w}) = C \sum_i \alpha_i^k. $$

\nid Furthermore, if $A$ is not nilpotent then there exists a number $N > 0$, such that for all $j > N$ and for all $\overline{w}$,  $t_k(j\BZ^r + \overline{w}) \neq 0$ for infinitely many values of $k$.
\end{lemma}

\subsubsection{$k$-covers and nilpotent quotients} Let $\CG$ be a residually torsion free finitely generated group. Let $\{\CG^{TF}\}_i$ be its torsion free lower central series. Given $x_1, \ldots, x_i \in \CG$, let $$[x_1, \ldots, x_i] = [\ldots [[x_1, x_2], x_3], \ldots, x_i]. $$

\nid Let $S$ be a finite generating set for $\Gamma$. It is a standard fact that for any j, $L_j = Z(\CG/ \CG_j^{TF})$ is generated by elements of the form $[a_1, \ldots, a_j]$ where $a_i \in S$. We require the following simple lemma. 

\begin{lemma} \label{nilmultlin}  Let $j,k \in \BN$. For any $1 \leq i \leq j$ and for any $a_1, \ldots, a_j \in \CG$: 
$$[a_1, \ldots, a_j]^k \equiv_{j} [a_1, \ldots, a_i^k, \ldots, a_j] $$
\nid where $\equiv_{j}$ is understood as having equal images in $L_{j}$. 
\end{lemma}

\begin{proof}
We prove the claim inductively on $j$. The claim is obvious for $j=1$. Assume we've proved the claim for all numbers up to $j$. We now prove it for $j+1$. Repeated application of the basic commutator identities: $[x,zy] = [x,y] \cdot [x,z]^y$ together with the fact that conjugation acts trivially on $L_i$ for all $i$ gives that: 

$$[a_1, \ldots, a_j^k] = [a_1, \ldots a_j] \cdot [a_1, \ldots a_j]^y \ldots \cdot [a_1, \ldots a_j]^{y^{k-1}} \equiv_j [a_1, \ldots, a_j]^k $$ 
For $i < j$ the inductive claim gives us a $w \in \CG^{TF}_{j}$ such that $$[a_1 \ldots, a_i^k, \ldots, a_j] = [[a_1, \ldots, a_i^j, \ldots, a_{j-1}], a_j] = [[a_1, \ldots, a_{j-1}]^kw, a_j] \equiv_j [[a_1, \ldots, a_{j-1}]^k, a_j]$$

Using the identity $[y,x] = [x,y]^{-1}$ and the fact that $[a_1, \ldots, a_j^k] \equiv_j [a_1, \ldots, a_j]^k$ now yields the result. 

\end{proof}

For any $k$, let $\CG[k]$ be the kernel of the natural map $\CG \to H_1(\CG, \BZ/ k\BZ)$. Let $\{\CG[k]_j^{TF} \}_j$ be the torsion free lower central series of $\CG[k]$, and let $$L_j[k] = Z(\CG[k]/ \CG[k]_j^{TF}).$$

\nid The inclusion $\CG[k] \to \CG$ induces a natural map $L_j[k] \to L_j$. As a corollary to Lemma \ref{nilmultlin}, we have the following. 
\begin{corollary} \label{knilp}
For any $j,k \in \BN$, the image of the natural map $L_j[k] \to L_j$ is $k^j L_j$.
\end{corollary}

\subsubsection{The upgrade lemma} Fix a vertex $u$ of $\CS^e \varphi$. For any integer $k$, let $G[k]$ be the kernel of the map $G \to H_1(G, \BZ/ k\BZ)$.  Let $F[k] = i^{-1}(G_k)$. Let $\pi_k: \Gamma_k \to \Gamma$ be the cover of $\Gamma$ corresponding to this map. The map $\varphi^k$ lifts to to $\Gamma_k$. Call this lift $\varphi_k$. Let $\CT_u^k$  be the vertex subgraph corresponding to $u$ in the transition graph of $\varphi^k$.
\begin{lemma} \tup{(The upgrade lemma)} \label{upgrade} Let $j \in \BN$, and suppose $u$ is a $(j+1)$-stable vertex for $\varphi$. Then for all but finitely many $k \in \BN$, the graph $\CT_u^k[\pi_k]$ is a $j$-stable extremal subgraph.

\end{lemma}

\begin{proof}
Let $\CT' \subseteq \CT_u$ be a $j$-vertex subgraph. As in Lemma \ref{consistentsubgraph}, let $\CM = V(\CT)^{L_{j+1}} \otimes \BR$, and let $\overline{\Delta}$ be the level $j+1$ base point free deviation function. Since the function $\overline{\Delta}$ is additive on cycles, we can extend it to in $\CT_u$, and produce a corresponding matrix $B \in M_{|V(\CT)|}(\BC[\CM])$ such that for any $k$: 

$$\trace(B^k) = \sum_\gamma \textbf{s}(\gamma) \overline{\Delta}(\gamma) $$

\nid where the sum is taken over all cycles of length $k$ in $\CT'$ and $\overline{\Delta}(\gamma)$ is understood as an element of $\BC[M]$. Since $V_{j+1}$ is torsion free, we have a natural inclusion map ${L_{j+1}}^V(\CT)  \subset M$. Let $m = |V(\CT)|$. Denote $\CL = {L_{j+1}}^V(\CT)   \cong (\BZ^{r})^m$ where $r = \textup{rank} (L_{j+1})$. Note that by construction, $\trace(B^k) \in \CL$ for every $k$.

\nid Since $\CT'$ is stable, the matrix $B$ is not nilpotent. Let $N$ be the number provided by Lemma \ref{top}. Fix $k > N$. For any $s$, let $\textbf{t}_{s}$ be the translation function corresponding to the map $G \to G / G^{TF}_{s}$, and $\textbf{t}[k]_{s}$ be the translation corresponding to the map $G[k] \to G[k] / G^{TF}[k]_{s}$. 

For any $i$, let $T_i = \sum_{\gamma} \textbf{s}(\gamma) \textbf{t}_{j+1}(\gamma)$, where the sum is taken over all cycles of length $i$ in $\CT'$. For any $x$ in the support of $T_i$, let $a^i_x$ be its coefficient. Since $\CT'$ is a $j$-vertex subgraph, for any $x, y$ in the support of $T_i$ we have that $xy^{-1} \in  L_{j+1}$. Let $T_i[x,k] = \sum_y a^i_y$ where the sum is taken over all $y$ such that $xy^{-1} \in k^{j+1} L_{j+1}$. 

Consider the sum $\sum_{\wt{\gamma}} \textbf{s}(\wt{\gamma})\textbf{t}[k]_{j+1}(\wt{\gamma})$, where the sum is taken over the lifts to $\CT'[\pi_k]$ of all cycles of length $i$ in $\CT'$. Pick $y$ in the support of this sum such that the image of $y$ in $G/G^{TF}_{j+1}$ is $x$. Let $\overline{y}$ be the image of $y$ in $G[k]/G^{TF}[k]_{j}$ under the natural inclusion map. 

By Corollary \ref{knilp}, the coefficient of $\overline{y}$ in $\sum_{\wt{\gamma}} \textbf{s}(\wt{\gamma}) \textbf{t}[k]_j(\wt{\gamma})$ is equal to $T_i[x,k]$. Since $\CT'$ is a $j$-vertex, this in turn is equal to $\trace(B^i)(\CL + \overline{w})$, for some $\overline{w} \in \CL$. By Lemma \ref{top}, this is not equal to $0$ for infinitely many values of $i$. This concludes the proof.

\end{proof}

\subsubsection{Weak vertex subgraphs and weakly consistent subgraphs} 
\begin{observation}
	We would like to use Lemma \ref{nilstab} together with repeated applications of Lemma \ref{upgrade} to produce a cover where a given vertex $u$ of $\CS^e \varphi$ is stable. However, Lemma \ref{upgrade} inputs a cover where a given \emph{vertex} is $(j+1)$-stable and outputs a further cover where a given \emph{face} is $j$-stable. It is quite possible that the vertices in this face are not themselves $j$-stable, which prevents the use of an inductive step. To circumvent this issue, we introduce a slight technical generalization called \emph{Weakly vertex subgraphs and weakly consistent subgraphs}.
\end{observation}

\begin{definition}
	A subgraph $\CT' \subset \CT$ is said to be \emph{of vertex type} if there exists some $v \in H_f \otimes \BR$ such that $\CT'$ consists all loops $\gamma$ with $\textbf{t}_n(\gamma) = v$.
\end{definition}

\begin{definition}
	 A subgraph $\CT' \subset \CT$ of vertex type is said to be a \emph{weak $j$-level vertex} if it is $j$-stable and there exists some linear transformation $T: H_f \otimes \BR \to \BR$ such that for any $v' \in  H_f \otimes \BR$ with $T(v') > T(v)$, the subgraph of vertex type corresponding to $v'$ is not $j$-stable. 
	 
\end{definition}

\begin{observation}
	Let $v, v'$ be as in the above definition. By replacing $f$ with a power of itself, we may assume that for all such $v'$, for any vertex $p$ and for any $k$:  $\sum \textbf{s}(\gamma_p)\textbf{t}_j(\gamma_p) = 0 $ where the sum is taken over all based loops of length $k$ in the vertex type subgraph corresponding to $v'$ that are based at $p$ and $\textbf{t}_j$ is the translation function corresponding to $N_j$.
\end{observation}

 \begin{definition}
 	A subgraph $\CT' \subseteq \CT$ is called \emph{weakly $j$-consistent} if for every vertex $p$ of $\CT'$ there exists $x_p \in N_j$ and an integer $d$ dividing the lengths of all cycles in $\CT'$ such that for all sufficiently large $k$:
 	
 	$$\sum_\gamma \textbf{s}(\gamma)\textbf{t}_j(\gamma) = C_k f^{k}(x_p)  \cdot \ldots \cdot f^{3d}(x_p) f^{2d}(x_p) f^{d}(x_p). $$
 	
 	  \nid where the above sum is taken in the group ring of $N_j$, the sum is taken over all based cycles of length $k$ in $\CT'$ that are based at $p$, $C_k$ is some number, and $\textbf{t}_j$ is the translation function corresponding to $N_j$.
 \end{definition}

	Suppose $\CT''$ is a weakly $j$-consistent subgraph of the weak $j$-level vertex subgraph $\CT'$ corresponding to $v \in H_f \otimes \BR$. For any vertex $p$, and any based loop $\gamma_p$ satisfying $\textbf{t}_n(\gamma_p) = v$,  define the deviation of $\gamma_p$ or $\Delta(\gamma_p)$ exactly as in Lemma \ref{consistentsubgraph}. The same calculation as the one done in Lemma \ref{consistentsubgraph} shows that $\Delta$ remains constant under cyclic reordering of based cycles based at $p$, and that it is additive on such cycles. Exactly as in Lemma \ref{consistentsubgraph} define the base point free deviation of a cycle $\gamma$ satisfying $\textbf{t}_n(\gamma) = v$. 
	
	We think of $\overline{\Delta}(\gamma)$ as an element of the group ring $\BC[L_{j}^{V(\CT)}]$. Extend the definition of $\overline{\Delta}$ to cycles $\gamma$ in $\CT''$ not satisfying $\textbf{t}_n(\gamma) = v$ by setting $\overline{\Delta}(\gamma) = 0$.

\begin{observation} \label{weakupgrade}
	Since $v$ is a $j$-level weak vertex subgraph, we can find a matrix $B \in M_{|V(\CT)|}(\BC[L_{j}^{V(\CT)}])$ such that for any $k$: 
	
	$$\trace(B^k) = \sum_\gamma \textbf{s}(\gamma) \overline{\Delta}(\gamma) $$
	
	where the sum is taken over all cycles of length $k$ in $\CT''$ satisfying $\textbf{t}_n(\gamma) = p$. The same proof as Lemma \ref{upgrade} now gives us that for all but finitely many $k$'s, the graph $CT''[\pi_k]$ is $(j-1)$-stable.
\end{observation}

\subsubsection{The vertex stabilization lemma} 

\begin{lemma} \tup{(The vertex stabilization lemma)} \label{stabilization}
For any vertex $u$ of $\CS^e \varphi$, there exists a solvable cover $\Gamma' \to \Gamma$ to which $\varphi$ lifts such that $ku$ is a stable vertex in $\Gamma'$ for some $k$.
\end{lemma}

\begin{proof}
	By Lemma \ref{nilstab}, the vertex $u$ is $j$-stable for some $j$. By Lemma \ref{upgrade}, we can find some $k_1$ cover $\pi_{k_1}$ such that the extremal subgraph $\CT_u[\pi_{k_1}]$ is $(j-1)$-stable. Since this extremal subgraph is $(j-1)$ stable, it has a weak $(j-1)$-level vertex $u_1$. Let $\CT_{u_1}$ be the corresponding graph. By applying Observation \ref{weakupgrade} we can find a $k_2$ and a $k_2$-cover $\pi_{k_2}$ of $\pi_{k_1}$ such that $\CT_{u_1}[\pi_{k_2}]$ is $(j-2)$-stable. This means that $\CT_u[\pi_{k_2}]$ is $(j-2)$ stable. 
	
	Proceeding inductively in this manner, we can find a cover $\pi$ in which $\CT_u$ is $1$-stable, and hence stable. Since this cover is obtained by iterating abelian covers, it is solvable. 
\end{proof}

\section{Proof of Theorems} \label{proofs}
\subsection{lemmas}
In this section we collect several technical lemmas that we require for our proof. 

\subsubsection{The cyclic deformation lemma}

\begin{lemma} \tup{(Cyclic deformation lemma)} \label{cyclic}
Let $\CT' \subseteq \CT$ be a weak $j+1$-stable vertex subgraph, that is weakly $j$-consistent. Let $\psi: G \to \BZ / q \BZ$ be a homomorphism into a cyclic group of prime order $q$ such that $i(F_n) \not \leq \ker \psi$ and such that $\CT'$ is $j$-stable in the cover corresponding to $\ker \psi$. Then for all sufficiently large primes $p$ there are homomorphisms $\psi_p: G \to \BZ / p \BZ$ such that $i(F_n) \not \leq \ker \psi_p$ and $\CT'$ is $j$ stable in the cover corresponding to $\ker \psi_p$. 

\end{lemma}

\begin{proof}  

Pick a basis $a_1, \ldots, a_l$ for $H_f$, and extend it to a minimal generating set $B$ for $H$ $H_1(G;\BZ)$ such that $\psi: H_1(G, \BZ) \to \BZ / q \BZ$ sends $a_1$ to $1$ and all other generators to $0$. For any prime $p$, let $\psi_p$ be the homomorphism $\psi_p: H_1(G, \BZ) \to \BZ / p \BZ$ sending $a_1$ to $1$ and all other generators to $0$. So $\psi = \psi_q$.

Let $L_{j+1} = G_{j+!}^{TF} / G^{TF}_{j+2}$. Denote by $L_{j+1}[\psi_p]$ the image of the $(j+1)^{th}$ term of the lower central series of $\ker \psi_p$ in $L_{j+1}$. Denote by $L_{j+1}^*$  the dual group of $L_{j+1}$ (that is - the group of all characters on $L_{j+1}$ with image in the unit circle.) Let $\chi  \in L_{j+1}^*$ such that $L_{j+1}[\psi_q] \leq \ker \chi$. 

We begin by showing that for all $\epsilon > 0$, there exist infinitely many primes $p$, and elements $\chi_p \in L_{j+1}^*$ such that $L_{j+1}[\psi_p] \leq \ker \chi_p$, and the distance from $\chi$ to $\chi_p$ is less than $\epsilon$ (here we're using the distance on $L_{j+1}^*$ induced by an embedding into $\BC^{\rank (L_{j+1})}$).

The group $L_{j+1}$ is generated by the images of elements of the form $[b_1, \ldots, b_{j+1}]$ where $b_i \in B$. By Lemma \ref{nilmultlin}, the lattice $L_{j+1}[\psi_p]$ is generated by elements of the form $p^s[b_1, \ldots, b_j]$ where $s$ is the number of times that $a_1$ appears in $b_1, \ldots, b_j$. 

Let $R = \tup{span}_{\BZ}([b_1, \ldots, b_{j+1}])$ be the lattice of formal linear combinations of generators as above. For any $i \geq 0$ let $Ri$ be the sublattice generated by all elements where $a_1$ appears $i$ times. Let $U_i$ be the image of $R_i$ in $L_{j+1}$. Then $L_{j+1}[\psi_p] = \sum p^i U_i$. 

By definition, we have that $U_0 \leq \ker \xi$, and must have that $U_0 \leq \ker \xi_p$ for every $p$. Given a lattice $L \leq \BZ^r$, we can find a direct sum decomposition $\BZ^r \cong \bigoplus M_i$, with $L \cong \bigoplus(L \cap M_i)$, and $L \cap M_i$ a finite index subgroup of $M_i$. Since every element of $L_{j+1}/L_{j+1}[\psi_p]$ has order that is a power of $p$, and $U_0 \leq L_{j+1}[\psi_p]$ for all $p$ we must have that $U_0$ is a direct summand of $L_{j+1}$. Write $L_{j+1} = U_0 \oplus V$. 

For every $p$, let $N_p \subset L_{j+1}^*$ be the set of all characters $\zeta$ such that $U_0 \leq  \ker \zeta$ and $\zeta(V)$ is contained in the set of $p^{th}$ roots of unity. Let $N_\infty$ be the set of all characters $\zeta$ such that $U_0 \leq \ker \zeta$. We have that $\xi \in N_\infty$, and any $\zeta \in N_p$ satisfies $L_{j+1}[\psi_p] \leq \ker \zeta$. We now conclude by noting that for every $\epsilon > 0$, the set $N_p$ is $\epsilon$-dense in $N_\infty$ for all sufficiently large $p$.

We now proceed similarly to the proof of Lemma \ref{upgrade}. Let $m = |V(\CT)|$, $\CL = \BC[L_{j+1}^V(\CT)]$. As in Lemma \ref{upgrade}, there is a $B \in M_m(\CL)$ such that for any $k$: $$\trace B^k = \sum_\gamma \textbf{s}(\gamma)\overline{\Delta}(\gamma)$$ where the sum is taken over all loops $\gamma$ of length $k$ in $\CT'$. Since $\CT'$ is $j+1$-stable, the matrix $B$ is not nilpotent. 

As in the proof of Lemma \ref{top}, let $p(x)$ be the characteristic polynomial of $B$, and let $\rho_1, \ldots, \rho_m$ be a collection of $m$ roots of $p(x)$. Since $\CT'$ is $j$-stable in the cover corresponding to $\ker \psi$, then as in Lemma \ref{upgrade}  there exists $\chi \in L_{j+1}^*$ with $L_{j+1}[\psi] \leq \ker \chi$,  and a root $\rho_i$ such that $\rho_i(\chi) \neq 0$. 

For every $k$, $\trace B^k[\chi] = \sum_i \rho_i(\chi)^k$. Thus, we can find a $k$ where $\trace B^k[\chi] \neq 0$. The polynomial $\trace B^k$ is continuous, and thus there exists an open set $U$ containing $\chi$ where $\trace B^k$ is non-zero at every point. For every $\chi' \in U$  there exists an $i$ such that $\rho_i(\chi') \neq 0$.

Let $N_{\psi_p} = \{\xi \in L_{j+1}^* | L_{j+1}[\psi_p] \leq \ker \xi \}$. By the above claim, there are infinitely many values of $p$ such that $N_{\psi_p} \cap U \neq \emptyset$. For each such $p$, there are infinitely many values of $l$ such that $$\sum_i \sum_{\chi' \in N_{\psi_p} } \rho_i(\chi')^l  \neq 0$$

Thus, as in the proof of Lemma \ref{upgrade}, the graph $\CT'$ is $j$ stable in the cover corresponding to $\ker \psi_p$. 
\end{proof}

\subsubsection{The cyclic cover multiplicity lemma}
\begin{lemma} \tup{(Cyclic cover multiplicity lemma)}\label{multiplicity} Let $\pi: \Gamma_p \to \Gamma$ be a cyclic cover of degree $p$, for a prime number $p$. Let $u$ be a vertex of $\varphi$ such that the image of $u$ in $\BZ/p\BZ$ is not $0$. Let $\CT_u$ be the vertex subgraph of $u$ and $A_{u}[\pi]$ be the matrix corresponding to $\CT_u[\pi]$. For any $k$, let $t^k[\pi,u]$ be the trace of the matrix $A_u[\pi]^k$. Then as an element of the additive group of $\BZ[H_1(\Gamma_p, \BZ)]$, $t^k(\pi,u)$ is divisible by $p$.
\end{lemma}

\begin{proof}
	For any integer $k$: $$\trace(A_u[\pi]^k) = \sum_\gamma \textbf{t}(\gamma) $$ where the above sum is taken over all \emph{based} cycles of length $k$ $\gamma$ in $\CT_u$. Since $\textbf{t}(\gamma)$ does not depend on the choice of base point of a cycle, we can rewrite the above as:
	$$\trace(A_u[\pi]^k)  = \sum_\delta n_\delta \textbf{s}(\delta) \textbf{t}(\delta) $$
	where the sum is taken over all unbased cycles of length $k$ in $\CT_u$ and $n_\delta$ is an integer. 
	
	Let $\delta$ be an unbased cycle. The group $\BZ$ acts transitively on the set of based cycles corresponding to $\delta$ by cyclic rotations. Let $s_\delta$ be the size of the image of $\BZ$ under this action. We call $s_\delta$ the \emph{cyclic stabilizer of } $\delta$. Then $n_\delta = \frac{k}{s_\delta}$. 
	
An alternative characterization of $n_\delta$ is the following. Let $\gamma$ be a based representative of $\delta$. Let $\gamma'$ be the minimal based subcycle of $\gamma$ (sharing the same base point) such that $\gamma = (\gamma')^l$. Then: $$n_\delta = \frac{k}{l} = \frac{k}{s_\delta} \tup{length}(\gamma').$$
	
	The based cycle $\alpha'$ projects to a based cycle $\alpha$ in $\CT_u$. By the definition of $\CT_u$, $\textbf{t}(\alpha) = \tup{length}(\alpha) u$. Since $u$ projects to a generator of $\BZ/p\BZ$, and $\gamma'$ is a cycle in the cover $\pi$, we must have that $\tup{length}(\alpha) = \tup{length}(\gamma')$ is divisible by $p$. By the above, this means that $n_\delta$ is divisible by $p$, as required.

\end{proof}

\subsubsection{Invariance of trace}
Let $\pi: \wt{\Gamma} \to \Gamma$ be a finite regular cover, to which $\varphi$ lifts to a map $\wt{\varphi}: \wt{\Gamma} \to \wt{\Gamma}$. Let $\calD$ be the deck group of the cover $\pi$. Since $\wt{\varphi}$ is a lift of $\varphi$, we have that $\wt{\varphi} \calD = \calD \wt{\varphi}$. Thus, some power of $\wt{\varphi}$ commutes with every element of $\calD$. Note that the group $\calD$ also acts on $H_{\wt{f}}$, the homology of the corresponding mapping torus.  

Given an edge $\eta$ of the train track graph of $f$, and a lift $\wt{\eta}$ of $\eta$ to $\wt{\Sigma}$, for any $\sigma \in \calD$ we have that $\wt{\varphi} \sigma(\wt{\eta}) = \sigma \wt{\varphi}(\wt{\eta})$. Thus, by the definition of $A_f[\pi]$ we have the following. 

\begin{lemma} \label{invtr} \tup{(Invariance of trace.)}
Let $\pi:  \wt{\Gamma} \to \Gamma$ be a finite regular cover with deck group $\calD$ to which $\varphi$ lifts. Then there is some integer $k$ such that $\trace(A_f[\pi]^k)$ is $\calD$ invariant. Furthermore, $h \in \tup{support}(\trace(A_f[\pi]^k))$ then the multiplicity of $h$ in $\trace(A_f[\pi]^k)$ is divisible by $n_h = \{\delta \in \calD | \delta_*h = h \}$.
\end{lemma}

\subsubsection{The polytope/lattice lemmas}
\begin{lemma} \tup{(The polytope/lattice lemma)} \label{polytopelattice}
Let $P \subset \BR^d$ be a $d$-dimensional convex polytope. There exists a lattice $L' \subset \BR^n$ and a translate $L$ of $L'$ such that $P \cap L$ consists of at least $d+1$ points, all of which are vertices of $P$, and the convex hull of the points in the intersection is a $d$-dimensional polytope.  Furthermore, for any vertex $v$ of $P$, $L$ can be chosen such that $v \in P \cap L$.
\end{lemma}

\begin{proof}

We will prove the claim  by induction on $d$. The claim is obvious for $d = 1$. Assume inductively that we've proved it for all dimensions up to $d - 1$. 

Fix a vertex $v$ of $P$. Without loss of generality, we can take $v$ to be the origin. and $v \in Q$. Let $F_1$ be a $d-1$ dimensional face of $P$ incident at $v$.  Let $F_2$ be the opposite face of $F_1$ (by this we mean that there exists a linear map $\omega: \BR^n \to \BR$ whose minimal value on $P$ is achieved precisely on $F_1$ and whose maximal value is achieved precisely on $F_2$.) 

The vertex $v$ is also a vertex of $F_1$. Let $W = \tup{Span}(F_1)$, and let $X$ be the set of all $\eta \in W^*$ whose minimum on $F_1$ is achieved exactly at $v$. The set $X$ is an open set. Pick a point $u \in F_2$, and let $T'$ be the linear operator $T'(x) = x - \omega(x)u$. The map $T'$ sends $F_2$ to $W$. For any $\eta \in X$, there is a face of $T'(F_2)$ on which $\eta \circ T'$ achieves its minimum. Since $X$ is open, we can choose $\eta \in X$ so this minimum is achieved at exactly one point. Call this point $v'$. 

Define a new projection operator $T(x) = x - \omega(x) v'$. Let $Q$ be the convex hull of $F_1 \cup T(F_2)$. This polytope is $d-1$ dimensional. Note that by construction, $v$ is a vertex of $Q$. Apply the induction assumption to $Q$ to get a lattice $L_0 \subset W$. Let $C_0 = L_0 \cap Q$. Let $L 
\subset \BR^d$ be the lattice generated by all the vertices in $T^{-1}(C_0)$. Note that since $C_0$ has at least $d$ points, and since $T(v') = v$, and $v \in C_0$, we get that $L$ intersects $P$ in at least $d+1$ points. By construction, $L$ is a lattice as required.

\end{proof}

\nid We require a slightly more specialized version of Lemma \ref{polytopelattice} that allows some more control over the vertices of $P$ that belong to $L$. 

\begin{corollary} \label{polytopelattice2}
	Let $P \subset \BR^d$ be a convex polytope, one of whose vertices is the origin. Suppose we have two maps $T: \BR^d \to \BR^m$ and $S: \BR^d \to \ker T$ such that for every vertex $v$ of $T(P)$, the set $T^{-1}(v)$ consists of a single vertex and such that the only vertices in $\ker S$ are vertices of the above form. Then there exists a lattice $L$ as in Lemma \ref{polytopelattice} such that $L \cap V(P)$ contains at least $m+1$ vertices that project to vertices of $T(P)$.
	
\end{corollary}

\begin{proof}
	Apply Lemma \ref{polytopelattice} to the polytopes $T(P)$ and $S(P)$. Let $C_S$ and $C_T$ be the corresponding sets of vertices. Let $L$ be the lattice generated by the following set. For every $0 \neq v \in C_S$, pick a single vertex in $S^{-1}(v)$, and then add all the vertices of the form $T^{-1}(C_T)$. This is a lattice as required. 
\end{proof}

\subsubsection{The positive vertices Lemma}
\begin{lemma} \tup{(Positive vertices Lemma)} \label{positivevertices} 
Suppose that every vertex of $\CS^e \varphi$ is stable. Then there exists a $k > 0$ such that for every vertex $v$ of $\CS^e \varphi$ with vertex matrix $A_v$: $\trace A_v^k = a_v  (kv)$ where $a_v >0$. 
\end{lemma}

\begin{proof}
By Lemma \ref{top}, for every $v$ there exists a collections of numbers $X_v = \{\alpha_1, \ldots, \alpha_r\}$ such that $\trace A_v^k = \sum \alpha_i^k$. Let $X = \bigcup_v X_v$. Let $$Y = \{\frac{\alpha}{|\alpha|} \mid 0 \neq \alpha \in X \}$$

\nid Enumerate the elements of $Y$:  $Y = \{\beta_1, \ldots, \beta_t\}$, and let $$\overline{y} = (\beta_1, \ldots, \beta_t) \in \BT^t = (S^1)^t$$ 

\nid There exists a number $M$ such that the $i^{th}$ component of $\overline{y}^M$ is $1$ for every $i$ where $\beta_i$ is a root of unity. Say $\beta_i = 1$ for $0 \leq i \leq r$. For every other $i$, the coordinate $\beta_i$ is not a root of unity. Thus, the set $\{\overline{y}^{Mj}\}_{j=1}^\infty$ is dense in the sub-torus $\BT^{t-r}$ obtained from $\BT^t$ by setting the first $r$ coordinates to $1$. 

The set of points in $\BT^{t-r}$ where all components have a positive real component is open. Thus, there exists a $j$ such that $\mathfrak{Re}(\beta_i^{Mj}) > 0$ for every $i$. Since each $\beta_i$ is of the form $\beta_i = \frac{\alpha_i}{|\alpha_i|}$, we get that $\mathfrak{Re}(\alpha_i^{Mj}) > 0$ for every $0 \neq \alpha \in X$. Set $k = Mj$

Since every vertex is stable, there is some $0 \neq \alpha \in X_v$ for every $v$. Since $\trace A_v^k = a_v(kv)$ for some integer $a_v$, we must have that $a_v > 0$ for every $v$, as required. 

\end{proof}

\subsubsection{The unbounded vertices in cyclic covers lemmas}
\begin{lemma} \tup{(Unbounded vertices for surface diffeomorphisms)} \label{unbounded} Let $f: \Sigma \to \Sigma$ be a pseudo-Anosov mapping class. Let $b$ be the number of boundary components of $\Sigma$. For ever prime $p$ there exists a finite cyclic cover $\pi: \wt{\Sigma} \to \Sigma$ of degree $p$ with $b$ boundary components and a lift $g$ of some power of $f$ to $\wt{\Sigma}$ such that $g_*: H_1(\wt{\Sigma}, \BR) \to H_1(\wt{\Sigma}, \BR)$ has eigenvalues off of the unit circle, or: $$ \#V(\CS^e g) \geq p - b -1 $$

\end{lemma}

\begin{proof}
	
	Choose a prime $p > 0$, and a homomorphism $\psi: \pi_1(\Sigma) \to \BZ/p\BZ$ such that every boundary component of $\Sigma$ is sent to $1 \in \BZ / p\BZ$. Let $\wt{\Sigma}$ be the cover corresponding to $\ker \psi$. Some power of $f$ lifts to the cover $\wt{\Sigma}$. By replacing with a further cover, we may assume that it commutes with the deck group. If its action on $H_1(\wt{\Sigma};\BR)$ does not have eigenvalues off of the unit circle, we can replace it with a further power such that all of the eigenvalues are $1$. By replacing $f$ with the relevant power, we may assume that all of the eigenvalues of  $f_*: H_1(\wt{\Sigma}; \BR) \to H_1(\wt{\Sigma};\BR)$ are also $1$. Call this lift $g$. By construction, $\wt{\Sigma}$ has the same number of boundary components as $\Sigma$. 
	
	Form the mapping torus $M_g$. The operator $g_*$ commutes with the action of the deck group on $H_1(\wt{\Sigma},\BR)$. Call this deck group $\calD$. For every $j$, the space $V_j = \ker(g_* - I)^j$ is $\calD$ invariant. The number $\dim V_{j+1} - \dim V_j$  is the number of Jordan block in the Jordan normal form of $g_*$ of size  greater than $j$. 
	
	Let $W = \ ker (H_1(\wt{\Sigma}, \BR) \to H_1(\Sigma,\BR)$. The spaces $V_j \cap W$ are all $\calD$-invariant, and have dimension divisible by $p-1$. Some of the spaces $V_j$ must intersect $W$ non-trivially by assumption. If $\dim (V_j \cap W) = \dim (V_{j+1} \cap W)$, then $W$ must contain $1$-eigenvectors. Thus, there must be at least $p-1$ such eigenvectors. If $\dim (V_j \cap W) < \dim (V_{j+1} \cap W)$ for some $j$, then 
	$\dim (V_{j+1} \cap W) - \dim (V_j \cap W)$ is divisible by $p$, and there must be at least $p$ Jordan blocks. Thus, we have that the multiplicity of $1$ as an eigenvalue of $g_*$ must be at least $p-1$. This proves the result by Lemma  \ref{dimshad}.

\end{proof}

\begin{lemma} \tup{(Unbounded vertices for fully irreducible autoharphisms)} \label{unbounded2} Let $f \in \tup{Aut}(F_n)$ be a train track representative of a fully irreducible automorphism. For ever prime $p$ there exists a finite cyclic cover $\pi: \wt{\Sigma} \to \Sigma$ of degree $p$  and a lift $g$ of some power of $f$ to $\wt{\Sigma}$ such that $g_*: H_1(\wt{\Sigma}, \BR) \to H_1(\wt{\Sigma}, \BR)$ has eigenvalues off of the unit circle, or: $$ \#V(\CS^e g) \geq p -1 $$

\end{lemma}

\begin{proof}
	
	Choose a prime $p > 0$, and a non-trivial homomorphism $\psi: \pi_1(\Sigma) \to \BZ/p\BZ$ . Let $\wt{\Gamma}$ be the cover corresponding to $\ker \psi$. Some power of $f$ lifts to the cover $\wt{\Gamma}$. By replacing with a further cover, we may assume that it commutes with the deck group. If its action on $H_1(\wt{\Gamma};\BR)$ does not have eigenvalues off of the unit circle, we can replace it with a further power such that all of the eigenvalues are $1$. By replacing $f$ with the relevant power, we may assume that all of the eigenvalues of  $f_*: H_1(\wt{\Gamma};\BR) ]\to H_1(\Sigma,\BR)$ are also $1$. Call this lift $g$. By construction, $\wt{\Sigma}$ has the same number of boundary components as $\Sigma$. Form the mapping torus $M_g$. The proof now proceeds exactly as in the previous lemma, concluding by using the fully irreducible automorphism part of Lemma  \ref{dimshad}.

\end{proof}

\subsection{Completing the proofs}
We begin by noting that in the proof of Theorem \ref{theorem3}, it is enough to prove the theorem for some power of $f$. Indeed, suppose $K \lhd F_n$ is a $f^k$-invariant subgroup for some $k$ such that $f^k_*: H_1(K; \BC) \to H_1(K;\BZ)$ has eigenvalues off of the unit circle. Let $K' = \bigcap \alpha(K)$, where the intersection is taken over all automorphisms $\alpha \in \tup{Aut}(F_n)$. The group $K'$ is characteristic in $F_n$, and if $F_n/K$ is solvable then so is $F_n/K'$. Since $K'$ is characteristic, $f(K') = K'$. 

The transfer map transfer map $T: H_1(K;\BC) \to H_1(k';\BC)$ is $f^k$-equivariant, and thus $f^k_*: H_1(K';\BC) \to H_1(K';\BC)$ has eigenvalues off of the unit circle. The same must then hold for $f$.

\nid We now prove Theorem \ref{theorem3} by dividing it into two cases.

\begin{prop} \label{pseudoanosov} \tup{(The pseudo Anosov case)}
Let $n \geq 2$ and let $\overline{f} \in \tup{Out}(F_n)$ be the image of a pseduo-Anosov mapping class. Let $f$ be a train track representative of $\overline{f}$. Then there exists a finite index subgroup $K \lhd F_n$ such that $f(K) = K$, $f_*: H_1(K;\BZ) \to H_1(F_n; \BZ)$ has eigenvalues off of the unit circle. Furthermore, we can choose $K$ such that $F_n/K$ is solvable.  

\end{prop}

\begin{proof}

If $f_*: H_1(\Sigma) \to H_1(\Sigma)$ has eigenvalues off of the unit circle we are done. If not, replace $f$ with a power of itself so that all of the eigenvalues of $f_*$ are $1$.

If $\CS^e \varphi$ has a non-stable vertex $v$, then by the stabilization Lemma (\ref{stabilization}), there exists a solvable cover $\pi: \Sigma' \to \Sigma$ to which $f$ lifts such that $v$ is stable in $\pi$. Since every solvable cover is the intersection of cyclic covers of prime order, we can find an intermediate cover $\Sigma' \to \Sigma'' \to \Sigma$ to which $f$ lifts such that $v$ is not stable in  the cover $\Sigma''$ and $\Sigma'$ is a cyclic cover of $\Sigma''$ of prime order. By the cyclic deformation lemma (\ref{cyclic}), there exist infinitely many primes $p$ and cyclic $p$-covers of $\Sigma''$ where $v$ is stable. Let $\pi_p: \Sigma_p \to \Sigma$ be such a cover corresponding to the prime $p$.

Since $v$ is not stable in $\Sigma''$, and by Observation \ref{invtr}, the support of $\trace A_v[\pi_p]^k$ contains an orbit of size $p$ for some $k$. By the cyclic multiplicity lemma (\ref{multiplicity}) the coefficient of every point in the support is divisible by $p$. Thus, the $L^2$ norm of $\trace A[\pi_p]^k$ is at least $\sqrt{p \cdot p^2} = p^{\frac{3}{2}}$. Suppose $\Sigma''$ is a cover of degree $M$. Then $\dim H_1(\Sigma_p; \BR) \leq Mp \dim H_1(\Sigma; \BR)$. In particular, if we take $p \gg 0$, we get that the $L^2$ norm of $\trace A[\pi_p]^k$ is greater than the first betti number of the corresponding mapping torus. This concludes the proof by the $L^2$ trace lemma (\ref{l2trace}). 

The same reasoning holds if $v$ is stable in the cover $\pi': \Sigma' \to \Sigma$, $\pi_p: \Sigma_p \to \Sigma'$ is a cyclic $p$ cover for $p \gg 0$ and the lift of $\CT_v$ to $\pi_p$ is a face graph that is not a vertex graph (because, once again, the trace will contain an orbit of size $p$).

For any cover $\pi: \Sigma' \to \Sigma$ to which $f$ lifts, let $d_\pi$ be the first betti number of the mapping torus, and let $$\gamma(\pi) = \frac{\sup_{k,L} \trace A[\pi]^k(L)}{d_\pi} $$

\nid where the supremum is taken over all integers $k$ and lattices $L$.

Let $\gamma_1$ be $\gamma$ of the trivial cover. If $\CS^e \varphi$ has any non-stable vertices we are done. If we can find a vertex $v$ such that for some sufficiently large prime $p$ there is cyclic cover $\pi_p$ such that $\CT_v[\pi_p]$ is not a vertex subgraph, we are also done. Otherwise, by the unbounded vertices for surface diffeomorphisms Lemma (\ref{unbounded}) we can find a prime $p \gg 0$ and a cyclic $p$-cover $\pi_p: \Sigma_p \to \Sigma$ to which $f$ lifts to a map $f_p$ such that $\CS^e f_p$ is at least $p - 1 - b$ dimensional, where $b$ is the number of boundary components of $\Sigma$. Furthermore, $\Sigma_p$ has the same number of boundary components as $\Sigma$. 

Every vertex of $\CS^e f$ is stable and lifts to a vertex of $\CS^e$. As before, if any of the vertices of $\CS^e f_p$ is not stable, we are done. Otherwise, by the positive vertices lemma (\ref{positivevertices}) we can replace $f$ with a power $f^k$ such that the coefficient of each vertex in $A[\pi_p]^k$ is positive. By observation \ref{invtr}, the coefficient of each vertex that is a lift from a vertex in $\Sigma$ is divisible by $p$. The coefficient of every other vertex is at least $1$. 

By the polytope/lattice Lemma and its corollary (\ref{polytopelattice} and \ref{polytopelattice2}) we can find a lattice $L$ such that $\trace A[\pi_p]^k \cap L$ consists of a set of at least $p-b$ vertices, $\dim \CS^e f + 1$ of which are lifts of vertices of $\CS^e f$. Thus, 

$$\gamma(\pi_p) \geq \frac{1}{p}\trace A[\pi_p]^k(L) \geq \gamma_1 + \frac{p-b - \dim(\CS^e f)}{p}$$ 

By taking $p \gg 0$, this can be made arbitrarily close to $\gamma_1 + 1$. Repeating the same argument for $\pi_p$. If it has unstable vertices, or vertices that lift to faces in cyclic covers then we are done. Otherwise, for any $\epsilon > 0$ we can find $p' \gg p$ and a cover $\pi_{p'}$ of $\Sigma_p$ such that $\gamma(\pi_{p''}) \geq \gamma_1 + 2 - \epsilon$. Iterating this process, we see that that the set $\gamma(\pi)$ is unbounded over all solvable covers $\pi$. This concludes the proof by the anchoring lemma (\ref{anchoring}).

\end{proof}

\begin{prop} \tup{(The fully irreducible case)}
Let $n \geq 2$ and let $\overline{f} \in \tup{Out}(F_n)$ be fully irreducible. Let $f$ be a train track representative of $\overline{f}$. Then there exists a finite index subgroup $K \lhd F_n$ such that $f(K) = K$, $f_*: H_1(K;\BZ) \to H_1(F_n; \BZ)$ has eigenvalues off of the unit circle. Furthermore, we can choose $K$ such that $F_n/K$ is solvable.  

\end{prop}

\begin{proof}
The proof here is nearly identical to the proof of Proposition \ref{pseudoanosov} except that instead of using the unbounded vertices lemma for surface diffeomorphisms, we use the unbounded vertices for fully irreducible autoharphisms Lemma (\ref{unbounded2}).

\end{proof}
\nid We now need only to deduce Theorem \ref{theorem1} from Theorem \ref{theorem3}.
\begin{proof}
To conclude the proof of Theorem \ref{theorem1}, we now need only address the case that $f$ has positive topological entropy but is not a pseudo-Anosov mapping class. By the Nielsen thurston classification, we can replace $f$ with a power of itself such that there exists a subsurface $\Sigma' \subset \Sigma$ such that $\Sigma'$ is $f$-invariant and $f$ restricted to $\Sigma'$ is a pseudo-Anosov mapping class.  Pick a base point $* \in \Sigma'$. 

By a theorem of Marshall Hall (see \cite{Hall}), there exists a finite index subgroup $K < \pi_1(\Sigma, *)$ such that $\pi_1(\Sigma',*) \leq K$, and $\pi_1(\Sigma',*)$ is a free subfactor of $K$. Replace $f$ with a power that fixes the subgroup $K$. The theorem now follows from applying Theorem \ref{theorem3} to $f|_{K'}$, and noting that $H_1(\Sigma'; \BC)$ injects into $H_1(K;\BC)$.

\end{proof}

\end{document}